\documentclass[12pt]{amsart}
\usepackage{CJKutf8}
\usepackage{tabls} 
\usepackage{url}
\usepackage[linktocpage = true]{hyperref}
\hypersetup{
 colorlinks = true,
 citecolor = blue}
\usepackage{color}
\definecolor{red}{rgb}{1,0,0}
\definecolor{green}{rgb}{0,1,0}
\definecolor{blue}{rgb}{0,0,1}
\definecolor{refkey}{gray}{.625}
\definecolor{labelkey}{gray}{.625}
\usepackage[lite]{amsrefs}
\usepackage[a4paper]{geometry}
\usepackage{amsfonts}
\usepackage{amssymb}
\usepackage{mathrsfs}
\usepackage{amsmath}
\usepackage{amsthm}
\usepackage{amscd}
\usepackage{enumerate}
\usepackage{paralist}
\usepackage{xypic}
\usepackage{graphicx}
\usepackage{mathtools}
\usepackage[all,cmtip]{xy}
\usepackage{kantlipsum}
\usepackage{tikz}
\usepackage{setspace}
\allowdisplaybreaks

\newcommand{\dce}{d_{\rm CE}}
\newcommand{\cech}{\delta}

\newcommand{\id}{\operatorname{id}}
\newcommand{\R}{\mathbb{R}}

\newcommand{\g}{\mathfrak{g}}

\newcommand{\bd}{\begin{displaymath}}
\newcommand{\ed}{\end{displaymath}}
\newcommand{\be}{\begin{equation}}
\newcommand{\ee}{\end{equation}}

\newcommand{\K}{\mathbb{K}}

\newcommand{\bsection}[1]{\Gamma(#1)}

\makeatletter
 \def\title@font{\normalsize\bfseries}
 \let\ltx@maketitle\@maketitle
 \def\@maketitle{\bgroup%
 \let\ltx@title\@title%
 \def\@\title{\resizebox{\textwidth}{!}{%
  \mbox{\title@font\ltx@title}%
 }}%
 \ltx@maketitle%
 \egroup}
\makeatother

\theoremstyle{plain}
\newtheorem*{zorn*}{Zorn's lemma}
\newtheorem*{tychonoff*}{Tychonoff's theorem}

\newtheorem{thm}{Theorem}[section]
\newtheorem{lem}[thm]{Lemma}
\newtheorem{Cor}[thm]{Corollary}
\newtheorem{Thm}[thm]{Theorem}
\newtheorem*{theorem*}{Theorem}
\newtheorem{Def}[thm]{Definition}
\newtheorem{prop}[thm]{Proposition}
\newtheorem{Ex}[thm]{Example}
\newtheorem{Rem}[thm]{Remark}

\begin{document}
\def\c{\mathscr{C}}
\def\ce{\mathrm{CE}}
\def\C{\mathbb{C}}
\def\CE{\mathrm{CE}}
\def\D{\mathcal{D}}
\def\E{\mathscr{E}}
\def\ev{\mathrm{ev}}
\def\F{\mathscr{F}}
\def\g{\mathfrak{g}}
\def\H{\textbf{H}}
\def\j{\mathscr{J}}
\def\L{\mathcal{L}}

\def\M{\mathcal{M}}
\def\m{\mathfrak{m}}
\def\t{\mathfrak{t}}
\def\O{\mathcal{O}}
\def\P{\mathcal{P}}
\def\r{\mathcal{R}}
\def\U{\mathcal{U}}
\def\v{\mathscr{A}}
\def\w{\mathfrak{W}}
\def\L{\mathcal{L}}
\def\X{\mathbb{X}}
\def\Y{\mathbb{Y}}
\def\spec{\text{spec}}
\def\Im{\text{Im}}
\def\coker{\operatorname{coker}}
\def\Ext{\operatorname{Ext}}
\def\End{\operatorname{End}}
\def\pr{\operatorname{pr}}
\def\id{\operatorname{id}}
\def\Der{\operatorname{Der}}
\def\Hom{\operatorname{Hom}}
\def\Jet{\operatorname{Jet}}
\def\Map{\operatorname{Map}}
\def\Mod{\operatorname{Mod}}
\def\sgn{\operatorname{sgn}}
\def\sh{\operatorname{sh}}

\newcommand{\structuresheaf}[1]{\mathcal{O}_{\mathcal{#1}}}
\newcommand{\CATderivationA}{\operatorname{dgDer}_\A}

\newcommand{\CATLeibnizoneA}{\operatorname{Leib}_\infty(\A)}

\newcommand{\CAThtyAmod}{\operatorname{H}(\operatorname{dg}\A)}

\newcommand{\Linfty}{ L_\infty  }

\newcommand{\piepartial}{\eth}

\newcommand{\inputvariable}{\cdot}

\newcommand{\Artinring}{\mathfrak{a}}

 \newcommand{\MCgA}{\textrm{MC}_\g(\v)}
\newcommand{\MCLPgA}{\textrm{MC}_{\Omega^\bullet_A(B)}(\v)}
\newcommand{\MCfunctor}{{\rm MC}_{\Omega^\bullet_A(B)}}

\newcommand{\algDefFctLA}{{\rm algDef}_{{(L,A)}}}
\newcommand{\weakDefFctLA}{{\rm wkDef}_{{(L,A)}}}

\newcommand{\firstDef}{{\rm Def}_\g}
\newcommand{\MCfunctorg}{{\rm MC}_\g}

\newcommand{\pairing}[2]{\langle #1,#2\rangle}
\newcommand{\rank}{\mathrm{rank}}

\newcommand{\Diff}{\mathrm{Diff}}
\newcommand{\CinfMK}{C^\infty(M,\K)}

\newcommand{\projection}{\mathrm{pr}}
\newcommand{\projectionA}{\mathrm{pr}_A}

\newcommand{\projectionB}{\mathrm{pr}_B}

\newcommand{\deltap}{\delta}
\newcommand{\deltas}{s}
\newcommand{\KScorrespondence }{\mathrm{KSc}}
\newcommand{\KSmap}{\mathrm{KS}}
\newcommand{\ArtCat}{\textbf{Art}}
\newcommand{\PhiA}{\Phi_A}
\newcommand{\PhiB}{\Phi_B}
\newcommand{\adLb}{{\mathrm{ad}_b  }}
\newcommand{\adLa}{{\mathrm{ad}_a  }}
\newcommand{\maximalidealofartin}{\m_\v}
 \newcommand{\bigiota}{\mathcal{I}}

\newcommand{\SpecofA}{\mathbb{P}}

\newcommand{\FINISH}{\textcolor{red}{FINISH}}

\title{Atiyah Classes in the Context of Generalized Complex Geometry}

\author{Dadi Ni}
\address{School of Mathematics and Statistics, Henan University} 
\email{\href{mailto:nidd@henu.edu.cn}{nidd@henu.edu.cn}}

\thanks{Research  supported by the Natural Science Foundation of Henan Province
 (No. 252300421766).}

\begin{abstract}
In analogy to the classical holomorphic setting, Lang, Jia and Liu introduced the notion of the Atiyah class for a generalized holomorphic vector bundle using three different approaches: leveraging $\rm{\check{C}}$ech cohomology, employing the first jet short exact sequence, and adopting the perspective of Lie algebroid pairs.  The purpose of this note is to establish the equivalence among these diverse definitions of the Atiyah class.
\end{abstract}

\maketitle

\begin{itemize}
	\item 
	{\it Keywords:} Atiyah class, Generalized complex structure.\\
\end{itemize}

\tableofcontents
\parskip = 0em 

\section{Introduction}
Generalized complex geometry was introduced by Hitchin \cite{Hitchin03}, and further developed by Gualtieri  \cite{Gualtieri} motivated by the study of mirror symmetry. It is common generalization of both symplectic and complex geometry. Another important class of generalized complex structures is holomorphic Poisson structures. Its applications include the study of $2$-dimensional supersymmetric quantum field theories, which occur in topological string theory, as well as compactification of string theory with fluxes. 

 A generalized complex structure on $M$ is 
an endomorphism $\mathcal{J}$
of the generalized tangent bundle $TM\oplus T^*M$ (rather
than the tangent bundle $TM$) that satisfies the condition $\mathcal{J}^2=-1$. Moreover, this endomorphism is orthogonal with respect to the standard nondegenerate symmetric bilinear form on $TM\oplus T^*M$  and  its $\pm i$-eigenbundles $L_+$ and $L_-$ are involutive with respect to the Courant bracket $TM\oplus T^*M$. It can be observed that both $L_+$ and $L_-$ qualify as Lie algebroids.  Using the notion of Courant algebroid given by Liu, Weinstein and Xu in \cite{LWX}, a generalized complex structure can be defined as a complex Dirac structure $L_+$ in $T_\C M\oplus T^*_\C M$ satisfying $L_+\cap \overline{L_+}=\{0\}.$ 

According to Gualtieri \cite{Gualtieri} and Hitchin \cite{Hitchin11}, a generalized holomorphic vector bundle over a generalized complex manifold $M$ is defined as a vector bundle $E$ with a Lie algebroid $L_-$-module structure. Note that the total space of such a bundle does not necessarily constitute a generalized complex manifold. The obstruction was discussed in \cite{Wang-JGP2011}. Later on, Lang, Jia and Liu \cite{DGA2023} introduced a distinct definition of generalized holomorphic vector bundle from a geometric point of view, ensuring that the total space $E$ is indeed a generalized complex manifold (see Definition~\ref{Definition3.1} for a detailed explanation). It turns out that this definition is a special case of Gualtieri's original concept. In this paper, we adopt the definition proposed by Lang, Jia and Liu.

Over the past two decades, extensive research has been conducted on generalized structures, employing diverse strategies and exploring their underlying mathematical frameworks.  The local structure was elaborated by Abouzaid-Boyarchenko \cite{Abouzaid-Boyarchenko}, and  was further  strengthened by Bailey \cite{Bailey}. Poon extended the notion of abelian complex structure to the context of generalized complex geometry \cite{Poon2021}. Grandini, Poon and Rolle \cite{Grandini-Poon-Rolle} studied deformations of generalized complex structures and their associated differential Gerstenhaber algebras . In \cite{MR3157903}, Wang examined the deformation theory of generalized holomorphic structures, allowing for variations in both the  generalized complex structure and the generalized holomorphic structure.

In this paper, we study the Atiyah class of a generalized holomorphic vector bundle. The Atiyah class of a holomorphic vector bundle $E$ over a complex manifold $M$, as initially introduced by Atiyah in \cite{Atiyah}, constitutes the obstruction to the existence of a holomorphic connection on said holomorphic vector bundle.  This Atiyah class plays a crucial role in the Rozansky-Witten theory \cite{Kapranov} and Kontsevich's work on deformation quantization \cite{Kontsevich}. Recently,  Chen, Sti\'{e}non and Xu introduced the notion of the Atiyah class for a Lie pair, i.e., a Lie algebroid with a Lie subalgebroid over the same base manifold \cite{CMP16}. 
Moreover, the Atiyah class of commutative dg algebras  and the twisted version  were studied in \cite{CLX}.  
We also refer to \cite{CLangX, Hong, MSX} for other versions of Atiyah classes in different settings.

Drawing inspiration from the method used to define the Atiyah class for a holomorphic vector bundle, Lang, Jia and Liu \cite{DGA2023} introduced three distinct approaches to define the Atiyah class for a generalized holomorphic vector bundle. Firstly, it is defined as a ${\rm \check{C}}$ech cohomology class $[\alpha_E]\in \check{\mathrm{H}}^1(M,G^*M\otimes \End(E))$, representing the obstruction to the existence of a generalized holomorphic connection (refer to Definition \ref{Def:Cech-Atiyah} for more details).
 For a generalized holomorphic
vector bundle $E$, Lang, Jia and Liu introduced the jet bundle $\mathfrak{J}^1 E$ and demonstrated that it fits into  a short exact sequence 
\[0\to G^*M\otimes E\to \mathfrak{J}^1 E\to E\to 0,\]
of generalized holomorphic vector bundles. The Atiyah class of $E$ is also characterized as the extension class
$$A(E)\in \Ext^1_{\mathcal{O}_M}(E, G^*M\otimes E)$$
of this short exact sequence (see Definition \ref{Def:Atiyah-jet} for more details). The third formulation of the Atiyah class arises from the point of view of Lie pairs. 
For a regular generalized complex manifold $M$, there exists a Lie pair $(T_\C M, \rho(L_-))$.  Based on the work \cite{CMP16} of Chen, Sti\'enon and Xu , the Atiyah class of a generalized holomorphic vector bundle can be defined as the Atiyah class $[R^\nabla]$ of the Lie pair $(T_\C M, \rho(L_-))$ and $\rho(L_-)$-module $E$.

In \cite{DGA2023}, Lang, Jia and Liu raised a question: What is the relation between these three distinct definitions? 
In this note, we will give a complete answer. We aim to  demonstrate the equivalence of these various definitions of the Atiyah class.

We investigate the ${\rm \check{C}}$ech cohomology and Lie algebroid cohomology theories on a generalized holomorphic vector bundle, establishing a fundamental relationship between these two cohomological frameworks. Using this correspondence, we construct a natural morphism 
$$\Pi_1\colon \check{\mathrm{H}}^1(M,G^*M\otimes \End(E)) \to \mathrm{H}^1(\rho(L_-), G^*M\otimes E),$$
which maps the ${\rm \check{C}}$ech cohomology class $[\alpha_E]$ to the Lie algebroid cohomology class $[R^\nabla]$.

Leveraging the fact that a generalized holomorphic vector bundle constitutes a $\rho(L_-)$-module, along with the work presented in \cite{Huebschmann90,CMP16}, we establish a natural injection $$\Psi\colon\Ext^1_{\mathcal{O}_M}(E, G^*M\otimes E)\to \mathrm{H}^1(\rho(L_-), G^*M\otimes E).$$ By utilizing the concept of the jet bundle arising from a Lie pair, we prove that the injection $\Psi$ maps the extension class $A(E)$ to  the Lie algebroid cohomology class $[R^\nabla]$. We also establish a canonical isomorphism from $\Ext^1_{\mathcal{O}_M}(E, G^*M\otimes E)$ to the ${\rm \check{C}}$ech cohomology group 
$\check{\mathrm{H}}^1(M,G^*M\otimes \End(E))$, and prove that this canonical map sends the extension class $A(E)$ to the ${\rm \check{C}}$ech cohomology class $[\alpha_E]$.
In other words, we obtain two cocycles representing the extension class.

In summary, we establish the following commutative diagram, which elucidates the interconnections among three distinct approaches to defining the Atiyah class:
\begin{equation*}
	\xymatrix@C=3pc@R=2pc{
		\Ext^1_{\mathcal{O}_M}(E, G^*M\otimes E) \ar[r]^{\Psi}\ar[d]_{\Phi}&\mathrm{H}^1(\rho(L_-), G^*M\otimes E)\\		\check{\mathrm{H}}^1(M,G^*M\otimes \End(E)).\ar[ur]_{\Pi_1}&
	}
\end{equation*}

Here is an outline of the present paper.
In Section $2$, we review the foundational aspects of generalized complex structures.  Section $3$  provides a comprehensive overview of generalized holomorphic vector bundles, along with detailed definitions of their $\rm{\check{C}}$ech cohomology and Lie algebroid cohomology. In Section $4$,  we establish a canonical identification between the Atiyah classes defined via the $\rm{\check{C}}$ech class and the Lie algebroid cohomology class associated with the Lie pair $(T_{\mathbb{C}} M,\rho(L_{-}))$. Section $5$ demonstrates the equivalence of Atiyah classes defined in terms of the extension classes of the first jet bundle $\mathfrak{J}^1E$ and their corresponding $\rm{\check{C}}$ech classes. Finally, in Section $6$, we prove that the Atiyah classes defined through the extension class and the cohomology class of the Lie pair $(T_{\mathbb{C}} M,\rho(L_{-}))$ are canonically identical.

\vskip 10pt

\noindent\textbf{Conventions and notations.}

\vskip 10pt

Throughout the paper, ``GC'' stands
for generalized complex, and ``GH'' stands for generalized holomorphic.  

\vspace{20pt}  

\section{Generalized complex manifolds}
We begin by revisiting the notion  of generalized complex (GC, for short) manifold, as presented in \cite{Gualtieri,DGA2023}. 
Let $M$ be a real $n$-dimensional smooth manifold. The direct sum $TM\oplus T^*M$ of its tangent and cotangent bundles is equipped with a canonical bilinear form which takes values in $C^\infty(M)$:
 \begin{equation}\label{pairing}
( X+\xi,Y+\eta) =\langle\xi, Y \rangle +\langle \eta, X\rangle,\qquad \forall X,Y\in \mathfrak{X}(M),\xi,\eta\in \Omega^1(M),
\end{equation} 
and a skew-symmetric bracket known as the \emph{Courant bracket}:
 \begin{eqnarray}\label{CB}
 [X+\xi,Y+\eta]:=[X,Y]+\mathcal{L}_X \eta-\mathcal{L}_Y \xi-\frac{1}{2}d(\langle \eta, X\rangle-\langle\xi, Y \rangle),
 \end{eqnarray}
 where $d$ denotes the de Rham differential on $M$. The Jacobi identity of this bracket fails, but it is controlled by a  coboundary:
\[[[e_1,e_2],e_3]+c.p.=\frac{1}{6}d\big(( [e_1,e_2],e_3)+( [e_2,e_3],e_1)+( [e_3,e_1],e_2)\big),\]
where $e_1,e_2,e_3\in \Gamma(TM\oplus T^*M)$ and c.p. denotes cyclic permutation.

\begin{Rem}
The vector bundle $TM\oplus T^*M$ together with bilinear form \eqref{pairing} and bracket \eqref{CB} is a Courant algebroid, which was introduced and well studied in \cite{LWX}. A {Dirac structure} of the Courant algebroid $TM\oplus T^*M$ is a maximal isotropic subbundle $A$ which is closed with respect to  Courant bracket \eqref{CB}. Obviously, a Dirac structure is a Lie algebroid. 
\end{Rem}

\begin{Def} \label{defgcg}\cite{Gualtieri}
A \textbf{GC structure} on $M$ is an endomorphism $\mathcal{J}:TM\oplus T^*M \to TM\oplus T^*M$ satisfying that 
\begin{itemize}
\item[\rm(1)] $\mathcal{J}^2=-1$;
\item[\rm(2)] $(\mathcal{J}e_1,\mathcal{J}e_2)=(e_1,e_2)$, for $e_1,e_2\in \Gamma(TM\oplus T^*M)$;
\item[\rm(3)] $\mathcal{J}$ is integrable; namely, the $\pm i$-eigenbundle $L_\pm\subset (TM\oplus T^*M)\otimes \mathbb{C}$ of $\mathcal{J}$ are closed under Courant bracket \eqref{CB}.
\end{itemize}
A manifold $M$ with a GC structure $\mathcal{J}$ is called a \textbf{GC manifold}. \end{Def}

A differential homeomorphism $f:(M,\mathcal{J}_M)\to (N,\mathcal{J}_N)$ between two GC manifolds is called a generalized holomorphic homeomorphism if  
\begin{equation*}
\begin{pmatrix} f_* & 0 \\0& (f^{-1})^* \end{pmatrix}\circ \mathcal{J}_M=\mathcal{J}_N\circ 
\begin{pmatrix} f_* & 0 \\0& (f^{-1})^* \end{pmatrix}.
\end{equation*}

Note that $L_-=\overline{L_+}$ and $L_+^*\cong L_{-}$ with respect to the pairing \eqref{pairing} on $T_{\mathbb{C}} M\oplus T_{\mathbb{C}}^* M=L_+\oplus L_{-}$. A  GC structure is in fact a complex Dirac structure $L_+$ in $T_{\mathbb{C}} M\oplus T_{\mathbb{C}}^* M$ satisfying that $L_+\cap \overline{L_+}=\{0\}$.  Note that both $L_+$ and $L_-$ constitute  Lie algebroids on $M$ with their
 anchor map $\rho$ defined as the natural projection $T_{\mathbb{C}} M\oplus T^*_{\mathbb{C}} M\to T_{\mathbb{C}} M$. 
We then have two Lie algebroid differentials
 \begin{eqnarray}\label{d+-}
 d_+:\Gamma(\wedge^\bullet L_+^*)\to \Gamma(\wedge^{\bullet+1} L_+^*),\qquad d_{-}:\Gamma(\wedge^\bullet L_{-}^*)\to \Gamma(\wedge^{\bullet+1} L_{-}^*).
 \end{eqnarray}
 In particular, for $w\in \Gamma(\wedge^k L_-^*)$, we have 
 \begin{eqnarray*}
 d_-(w)(u_1,\cdots,u_{k+1})&=&\sum_{i=1}^{k+1}(-1)^{i+1}\rho(u_i)w(u_1,\cdots,\hat{u_i},\cdots,u_{k+1})\\ &&+\sum_{i<j}(-1)^{i+j} 
 w([u_i,u_j],u_1,\cdots,\hat{u_i},\cdots,\hat{u_j},\cdots,u_{k+1}),
 \end{eqnarray*}
  for $u_1,\cdots,u_{k+1}\in \Gamma(L_-)$.
  
  For any $f\in C^\infty(M)$, we have $df=d_+f+d_-f$, 
where  $d_+$ and $d_-$ are defined in Equation \eqref{d+-} and $d$ is the de Rham differential on $M$. Furthermore, there is  a Poisson structure on $M$ given by 
\begin{eqnarray*}
\{f,g\}_M=(d_{+} f,d_{-} g),\qquad \forall f,g\in C^\infty(M),
\end{eqnarray*}
where $(\cdot,\cdot)$ is the pairing in Equation \eqref{pairing}.

\begin{Rem}
By \cite[Theorem 2.6]{LWX}, two transversal Dirac structures of a Courant algebroid constitute a Lie bialgebroid.
So the pair $(L_+,L_-)$ is a Lie bialgebroid.
\end{Rem}

\begin{Ex}
Complex and symplectic structures are two extreme cases of generalized complex structures. Holomorphic Poisson structures provide another type of generalized complex structures. See \cite[Example 2.3, 2.4, 2.5]{DGA2023} for details.
\end{Ex}

An endomorphism  $\mathcal{J}$ on $TM\oplus T^*M$ satisfying $\mathrm{(1)}$ and $\mathrm{(2)}$ in Definition \ref{defgcg}  can be written  in block form as 
\begin{equation*}
\mathcal{J}=\begin{pmatrix} J & \beta \\ B& -J^* \end{pmatrix},\qquad J\in \End(TM), B\in \Omega^2(M), \beta\in \mathfrak{X}^2(M).
\end{equation*}
 An \emph{$B$-field transformation}  (or, $B$-transform) is  defined to be an orthogonal automorphism, denote by $\exp(B)$, of $TM\oplus T^*M$ given by a real closed $2$-form $B\in \Omega^2(M)$ via:  
 \[X+\xi\mapsto X+\xi+\iota_X B,\qquad \forall X\in \mathfrak{X}(M),\xi\in \Omega^1(M).\]
If $\mathcal{J}$ defines a GC structure on $M$, then $\mathcal{J}'=\exp(B)\mathcal{J}\exp(-B)$ is another GC structure on $M$, which follows from the fact that $\exp(B)$ preserves the Courant bracket on $TM\oplus T^*M$ \cite{Gualtieri}. Moreover, the $\pm i$-eigenbundle of $\mathcal{J}'$ is given by $L'_\pm=\exp(B)(L_\pm).$

\begin{Def}
A function $f\in C^\infty(M)$ on a generalized complex manifold $(M,\mathcal{J})$ is called a {generalized holomorphic function} if  $d_{-} f=0$.
\end{Def}
Denote by $\mathcal{O}_M$ the sheaf of local generalized holomorphic functions on $M$. By definition, $\mathcal{O}_M=\ker d_-\subset C^\infty(M)$ is a subsheaf of the sheaf of smooth functions on $M$.

Let $E$ denote the distribution $\rho(L_+)$. Note that $E\cap \overline{E}$ is a distribution in $T_\C M$ which is stable under complex conjugation, and hence has the form $\Delta\otimes_\R \mathbb{C}$ for some distribution $\Delta\subset TM$. We refer to a point $m$ in $M$ as a regular point if the real dimension of $\Delta$ remains constant in a neighborhood of $m$. If every point in $M$ possesses this regularity, we term $M$ as a \textbf{regular GC manifold}. For example, if the GC sructure on $M$ is complex, then $\Delta=0$, whereas if the GC structure on $M$ is symplectic, then $\Delta=TM.$

\begin{Rem}
When $M$ represents a regular GC manifold, the pair $(T_\C M, \rho(L_-))$  constitutes a Lie algebroid pair, or simply a Lie pair.
\end{Rem}

Next, we will pay our attention to the local properties of GC manifolds. The initial insight into the local structure of regular GC manifolds is attributed to Gualtieri \cite{Gualtieri}. Subsequently, for a general GC manifold, the local structure was elaborated by Abouzaid-Boyarchenko \cite{Abouzaid-Boyarchenko}, and  was further  strengthened by Bailey \cite{Bailey}.

Gualtieri's generalized Darboux Theorem, presented in \cite{Gualtieri}, states the following: At a regular point $m$ of type $k$ on a GC manifold $(M,\mathcal{J})$, there exists an open neighborhood $U$ such that, up to a $B$-transform, $U$ is diffeomorphic to the product of an open subset $V$ of $\mathbb{C}^k$  and an open subset $W$ of $\mathbb{R}^{2n-2k}$. Specifically, there exists a closed $2$-form $B\in \Omega^2(V\times W)$ such that the generalized holomorphic homeomorphism
\begin{eqnarray*}\label{localtri}
\varphi: (U,\mathcal{J})\cong (V\times W,\exp({B})(\mathcal{J}_{J_0}\times \mathcal{J}_{\omega_0} )\exp({-B}))
\end{eqnarray*}
holds, where $J_0$ is the canonical complex structure on $\mathbb{C}^k$ and $\omega_0$ is the standard symplectic structure on $\mathbb{R}^{2n-2k}$.
Given a coordinate system $z=(z_1,\cdots,z_k)$   for $\mathbb{C}^k$  and a coordinate system $(p,q)=(p_1,\cdots,p_{n-k},q_1,\cdots,q_{n-k})$ for $\mathbb{R}^{2n-2k}$, we refer to the corresponding local coordinates
\begin{eqnarray}\label{canon}
(U,\varphi; z,p,q)=(U,\varphi; z_1,\cdots,z_k, p_1,\cdots,p_{n-k}, q_1,\cdots,q_{n-k})
\end{eqnarray} as the \emph{canonical coordinates}.

\begin{prop}\cite{DGA2023}\label{coordinates2}
Let $(M,\mathcal{J})$ be a regular GC manifold. Suppose that  $(U,\varphi;z,p,q)$ and $(U',\psi;z',p',q')$
are two canonical coordinates as in Equation \eqref{canon} and $U\cap U'\neq \varnothing$. Then the coordinate transformation $\phi:=\psi\circ \varphi^{-1}:\varphi(U\cap U')\to \psi(U\cap U')$ 
 satisfies that
\[\frac{\partial z'_\lambda}{\partial \bar{z}_\mu}=0,\qquad \frac{\partial z'_\lambda}{\partial p_\nu}
=\frac{\partial z'_\lambda}{\partial q_\nu}=0,\qquad\lambda,\mu=1,\cdots,k; \nu=1,\cdots,n-k.\]
\end{prop}

Let $(M,\mathcal{J})$ be a regular GC manifold.
 A function $f:M\to \mathbb{C}$ is a generalized holomorphic function if and only if  for any $m\in M$ of type $k$, when expressed in the canonical coordinates $(z,p,q)$ as defined in Equation \eqref{canon}, the following conditions hold:
\[\frac{\partial f}{\partial \bar{z}_\lambda}=0,\qquad \frac{\partial f}{\partial p_\mu}=\frac{\partial f}{\partial q_\mu}=0,\qquad \lambda=1,\cdots,k; \mu=1,\cdots,n-k.\]

\section{Cohomology on generalized  holomorphic vector bundles}
We will recall the concept of generalized holomorphic vector bundles in a manner analogous to that of holomorphic vector bundles, utilizing local trivialization. While this definition is more rigorous than the one presented by Gualtieri in \cite{Gualtieri}, it possesses its own distinct advantages.
\begin{Def}\cite{DGA2023}\label{Definition3.1}
Suppose that  $M$ is a GC manifold. A complex vector bundle $\pi:E\to M$ is called a \textbf{generalized holomorphic vector bundle} (GH vector bundle, for short), if 
\begin{itemize}
\item[\rm{(1)}] $E$ is a GC manifold;
\item[\rm{(2)}] there is an open cover $\{U_i\}_{i\in I}$ of $M$ and a family of local trivializations of the vector bundle $E$ \[\{\varphi_i: E|_{U_i}=\pi^{-1}(U_i)\to U_i\times \mathbb{C}^r\}_{i\in I}\] satisfying that
$\varphi_i$ for each $i$ is a  GH homeomorphism, where $U_i\times \mathbb{C}^r$ is associated with the standard product GC structure.
\end{itemize}
\end{Def}

\begin{prop}\cite{DGA2023}\label{prop:Explicit}
Let $E$ be a complex vector bundle over $M$ with a family of local trivializations and the corresponding transition functions as follows:
\[\{\varphi_i:\pi^{-1}(U_i)\to U_i\times \mathbb{C}^{r}\},\qquad \varphi_{ij}=\varphi_i\circ \varphi_j^{-1}:U_i\cap U_j\to \mathrm{GL}(r,\mathbb{C}).\] Then $E$ is GH vector bundle over $M$ with local trivializations $\{\varphi_i\}$ 
 if and only if each entry $A_{\lambda\mu}: U_i\cap U_j\to \mathbb{C}$ of $\varphi_{ij}=(A_{\lambda\mu})_{r\times r}$  is a GH function.
\end{prop}

Next, we provide an example about GH vector bundles arising from a regular GC manifold.
Let $(M,\mathcal{J})$ be a regular GC manifold. Then $G^*M:=L_{-}\cap T^*_{\mathbb{C}} M$ is a GC vector bundle over $M$ \cite{DGA2023}, which is called the \textit{GH cotangent bundle} of $M$.

In complex geometry, a holomorphic vector bundle admits an  $\bar{\partial}_E$ operator. Such an operator also exists for a GH vector bundle. The following theorem formalizes this observation:
\begin{Thm}\cite{DGA2023}\label{Thm:NN}
If $E$ is a GH vector bundle over a GC manifold $(M,\mathcal{J})$, there exists an $L_-$-connection $\bar{\partial}_E$ on $E$ such that $\bar{\partial}_E^2=0$.
\end{Thm}

A section $s$ of a GH vector bundle $E$ is termed GH if it satisfies condition $\overline{\partial}_E s=0$. Let  $\mathcal{O}(E,U)$  represent the  space of local GH sections of $E$ on the open subset $U$. Similarly, $\Gamma(E,U)$ denotes the space of local smooth sections of $E$ on $U$.

In other words,  this theorem implies that a GH vector bundle $E$ can be regarded as an $L_-$-module (for detailed definitions concerning Lie algebroid modules, refer to \cites{CMP16, DGA2023}). Furthermore, when $M$ is a regular GC manifold, $E$ acquires the structure of an $\rho(L_{-})$-module, with the action defined as follows:
\[\nabla_{\rho(l)} e:=\langle l, \overline{\partial}_E(e)\rangle,\qquad \forall l\in \Gamma(L_{-}), e\in \Gamma(E).\] 
This definition is well-defined because for any $\xi\in \Gamma(L_{-})\cap \Gamma(T^*_{\mathbb{C}} M)$, by definition, we have 
$\langle \xi,\overline{\partial}_E e\rangle =0$.


Denote by $\dce$ the Lie algebroid differential associated with the $\rho(L_{-})$-module structure on $E.$ This gives rise to a complex $(C^\bullet(\rho(L_{-}),E),\dce)$, where  $C^k(\rho(L_{-}),E)=\Gamma(\wedge^k \rho(L_{-})^*\otimes E)$ represents the space of $k$-cochains and the differential 
\[\dce: C^k(\rho(L_{-}),E)\longrightarrow C^{k+1}(\rho(L_{-}),E),\]
is explicitly defined by the following formula for $\omega\in C^k(\rho(L_{-}),E)$ and $a_i\in \Gamma(\rho(L_{-}))$:
\begin{eqnarray*}
\dce \omega(a_1,\cdots,a_{k+1})&=&\sum_{i=1}^{k+1}(-1)^{i+1} \nabla_{a_i} \omega(a_1,\cdots,\widehat{a_i},\cdots,a_{k+1})\\ && +\sum_{i< j} (-1)^{i+j}\omega([a_i,a_j],\cdots,\widehat{a_i},\cdots,\widehat{a_j},\cdots,a_{k+1}).
\end{eqnarray*}
The cohomology  of this complex is called the \emph{Lie algebroid cohomology}  of $\rho(L_{-})$ with coefficients in $E$, denoted by $\mathrm{H}^\bullet_{\rm CE} (\rho(L_{-}),E)$.

Next, we briefly review the notion of $\rm{\check{C}}$ech cohomology. Let $\mathcal{F}$ be a sheaf on a manifold $M$. 
Fix an open covering $\mathcal{U}=\{ U_i\}_{i\in \Lambda}$  with $\Lambda$ an ordered set. Denote by $U_{i_0\cdots i_k}:=U_{i_0}\cap \cdots \cap U_{i_k}$. Then set
\[C^k(\mathcal{U},\mathcal{F}):=\Pi_{i_0<\cdots < i_k} \Gamma(U_{i_0\cdots i_k},\mathcal{F}).\]
There is a natural differential 
\[\cech\colon C^k(\mathcal{U},\mathcal{F})\longrightarrow C^{k+1}(\mathcal{U},\mathcal{F}),\qquad \alpha= (\alpha_{i_0\cdots i_k})\mapsto \cech\alpha,\]
with \[(\cech\alpha)_{i_0\cdots i_{k+1}}=\sum_{j=0}^{k+1}(-1)^k\alpha_{i_0\cdots \hat{i_j}\cdots i_{k+1}}|_{U_{i_0\cdots i_{k+1}}}.\]
One obtains a complex $(C^\bullet(\mathcal{U},\mathcal{F}),\cech)$ by the fact that $\cech^2=0$. The cohomology  of this complex is called the \emph{$\rm{\check{C}}$ech cohomology} of the sheaf $\mathcal{F}$ with respect to the fixed covering $\mathcal{U}$.

Consider the sheaf of GH sections of $E$, then we have the $\rm \check{C}$ech cohomology of $E$ with respect to a covering $\mathcal{U}$, denoted by $\check{\mathrm{H}}^k(\mathcal{U},E)$. 
In the subsequent discussion, we aim to  establish  a fundamental relationship between these two
cohomology.

\begin{prop}\label{Thm:Cech-D}
For every GH vector bundle $E$   over a regular GC manifold $M$ and every locally finite covering $\mathcal{U}=\{U_i\}_{i\in\Lambda}$, there exists a natural morphism of vector spaces: 
$$\Pi_k \colon \check{\mathrm{H}}^k(\mathcal{U},E) \to \mathrm{H}^k_{\rm CE}(\rho(L_-), E),\quad k\geq 0.$$
\end{prop}

\begin{proof}
Let $h_i\colon M\to \C, i\in\Lambda,$ be a partition of unity subordinate to the covering $\mathcal{U}=\{U_i\}_{i\in\Lambda}$. Given an element $\xi=(e_{c_0\cdots c_k})\in C^k(\mathcal{U},E),$ we define two mappings:
$$\phi_i(\xi)=\sum_{c_1\cdots c_k}\dce h_{c_1}\wedge\cdots \wedge\dce h_{c_k}\otimes e_{ic_1\cdots c_k}\in \Gamma(\Lambda^k \rho(L_-)^*\otimes E, U_i),$$
and
$$\phi (\xi)=\sum_{i\in\Lambda} h_i \phi_i(\xi) \in \Gamma(\Lambda^k \rho(L_-)^*\otimes E).$$
Since each $e_{c_0\cdots c_k}$ is a GH section, it follows directly from the properties of GH sections and the definition of the Lie algebroid differential that $\dce \phi_i(\xi)=0$. Consequently, we have
$$\dce \phi (\xi)=\sum_i \dce h_i\wedge \phi_i(\xi)=\sum_{c_0\cdots c_k}\dce h_{c_0}\wedge\cdots \wedge\dce h_{c_k}\otimes e_{c_0\cdots c_k}.$$

Now, we proceed to verify that $\phi$ is a morphism of complexes by performing a direct computation. We have
\begin{align*}
\phi(\cech \xi)
&=\sum_{i\in\Lambda} h_i\sum_{c_0\cdots c_k} \dce h_{c_0}\wedge\cdots\wedge\dce h_{c_k}\otimes (\cech\xi)_{ic_0\cdots c_k} \\
&=\sum_{i\in\Lambda} h_i\big(\dce\phi(\xi)+\sum_{j=0}^k (-1)^{j+1}\sum_{c_0\cdots c_k} \dce h_{c_0}\wedge\cdots\wedge \dce h_{c_k}\otimes\xi_{ic_0\cdots \hat{c_j}\cdots c_k} \big)\\
&=\sum_{i\in\Lambda} h_i\big(\dce\phi(\xi)\big)\\
&=\dce\phi(\xi).
\end{align*}
It follows that $\phi$ is a morphism of complexes:
\begin{equation*}
	\xymatrix@C=2.5pc@R=1.5pc{
		\cdots \ar[r]^{\cech}&C^k(\mathcal{U},E)\ar[r]\ar[d]^{\phi}\ar[r]^{\cech}
		&C^{k+1}(\mathcal{U},E)\ar[r]^{\cech}\ar[d]^{\phi}
		&\cdots\\		\cdots\ar[r]^{\dce \qquad}&\Gamma(\Lambda^k \rho(L_-)^*\otimes E)\ar[r]^{\dce}
		&\Gamma(\Lambda^{k+1} \rho(L_-)^*\otimes E)\ar[r]^{\qquad\dce}
		&\cdots.
	}
\end{equation*}
Finally, let $\Pi_\bullet$ denote the morphism induced by $\phi$ in cohomology.
\end{proof}

It can be shown through a straightforward  argument that $\Pi_k$ is independent of the choice of the partition of unity. Moreover, we have the following:
\begin{prop}\label{Prop:Cech-D}
Let $\mathcal{U}=\{ U_i\}_{i\in\Lambda}$ be a locally finite covering of $M.$ The morphism
$$\Pi_1\colon \check{\mathrm{H}}^1(\mathcal{U},E)\to \mathrm{H}^1_{\rm CE}(\rho(L_-), E),$$
defined in Proposition \ref{Thm:Cech-D}, is an injection.
\end{prop}

\begin{proof}
Let $h_i\colon M\to \C, i\in\Lambda,$ be a partition of unity subordinate to the covering $\mathcal{U}=\{U_i\}_{i\in\Lambda}$.
If $\Pi_1[(e_{ij})]=0$, there exists some $e\in\Gamma(E)$ such that $\sum_i h_i \sum_j \dce h_j\otimes e_{ij}=\dce e.$ Since $(\{U_{ij}, e_{ij} \})$ is closed, over $U_i\cap U_k$, we have
$$\sum_j \dce h_j\otimes e_{ij}-\sum_j \dce h_j\otimes e_{kj}=\sum_j \dce h_j\otimes e_{ik}=0. $$ Therefore all the $\sum_j \dce h_j\otimes e_{ij}$ come from the global section $\dce e\in\Gamma(\rho(L_-)^*\otimes E),$ i.e., on $U_i$, we have $\sum_j \dce h_j\otimes e_{ij}=\dce e|_{U_i}.$ Obviously, $\sum_j  h_j e_{ij}- e|_{U_i}\in \mathcal{O}(E, U_i).$ Moreover, we have $$(\sum_j  h_j e_{ij}- e|_{U_i})-(\sum_j  h_j e_{kj}- e|_{U_k})=e_{ik},$$
which shows that $\{U_{ij}, e_{ij} \}$ is exact. 
It follows that $\Pi_1$ is an injection.
\end{proof}

\begin{Rem}
In complex geometry, when the covering possesses favorable properties, Leray's theorem establishes an isomorphism between $\rm \check{C}$ech cohomology and Dolbeault cohomology. However, in the context of generalized complex manifolds, despite the covering maintaining these good properties, we cannot establish such an isomorphism. This is due to the failure of the Poincaré lemma.
\end{Rem}

\begin{Rem}\label{Rem:D-Cech}
Given $\omega \in\Gamma(\rho(L_-)^*\otimes E)$ such that $\dce \omega=0$, and assuming that the restriction $\omega|_{U_i}$ is $\dce$-exact on each open set $U_i$, our objective is to demonstrate the existence of an element $\xi\in \check{\mathrm{H}}^1(\mathcal{U},E)$ for which $\Pi_1(\xi)=[\omega]$, where $[\omega]$ denotes the cohomology class of $\omega$.

To this end, for each open set $U_i$, since $\omega|_{U_i}$ is $\dce$-exact, there exists a section $e_i\in\Gamma(E,U_i)$ such that $\dce e_i=\omega|_{U_i}$. Similarly, for another open set $U_j$, there exists a section $e_j\in\Gamma(E,U_j)$ satisfying $\dce e_j=\omega|_{U_j}$.  Whenever $U_i\cap U_j:=U_{ij}\neq \emptyset$, we observe that $e_{ij}=e_i-e_j$ defines a section in $\mathcal{O}(E,U_{ij})$.

We claim that the  1-cocycle $\Pi_1 (\{U_{ij}, e_i-e_j\})$ represents the cohomology class $[\omega]$. To verify this claim, we perform the following computation:
\begin{align*}
&\sum_i h_i \sum_j \dce h_j\otimes (e_i-e_j)\\
&\qquad=-\sum_j \dce h_j\otimes e_j\\
&\qquad=-\dce(\sum_j h_j e_j)+\sum_j h_j\otimes \dce e_j\\
&\qquad=-\dce(\sum_j h_j e_j)+\omega,
\end{align*}
where the first equality follows from the property $\sum_{i\in\Lambda}h_i=1$ of the partition of unity, the second equality is obtained by applying the Leibniz rule for the Lie algebroid differential $\dce$.
\end{Rem}

\section{The Atiyah class in $\rm \check{C}$ech  and Lie algebroid cohomology} 
We shall recall the notion  of GH connections on a GH vector bundle. The obstruction class for the existence of such a connection is  referred to as the Atiyah class of this GH vector bundle.
\begin{Def}
Let $E$ be a GH vector bundle over  a regular GC manifold $(M,\mathcal{J})$. A \textbf{GH connection} on $E$ is a $\mathbb{C}$-linear map (of sheaves) $D:E\to G^*M\otimes E$ such that
\[D(fs)=d_{+}f\otimes s+fD(s),\]
for all local GH functions $f$ on $M$ and all local GH sections $s$ of $E$, where $d_+$ is the Lie algebroid differential of the $+i$-eigenbundle  $L_+$ of $\mathcal{J}$ as in Equation \eqref{d+-}.
\end{Def}


With respect to the trivialization $\varphi_i: \pi^{-1}(U_i)\to U_i\times \mathbb{C}^r$ in Definition \ref{Definition3.1} on $E$, we can express a local GH connection on $U_i\times \mathbb{C}^r$ as $d_+ +A_i$ , where $A_i$ is a matrix-valued GH $1$-form on $U_i$. 
These local connections can be glued into a global connection on $E$ if and only if the following equality holds on $U_{ij}$:
\[\varphi_i^{-1}\circ (d_++A_i)\circ \varphi_i=\varphi_j^{-1}\circ (d_++A_j)\circ \varphi_j.\]
This is equivalent to
\begin{eqnarray*}
\varphi_j^{-1}\circ A_j\circ \varphi_j-\varphi_i^{-1}\circ A_i\circ \varphi_i&=&\varphi_j^{-1}\circ (\varphi_{ij}^{-1}\circ d_+\circ \varphi_{ij}-d_+)\circ \varphi_j\\ &=&\varphi_j^{-1}\circ (\varphi_{ij}^{-1}d_+ (\varphi_{ij}))\circ \varphi_j,
\end{eqnarray*}
where $\varphi_{ij}=\varphi_i\circ \varphi_j^{-1}$. Due to the relation $\varphi_{ij}\circ \varphi_{jk}\circ \varphi_{ki}=1$, the left-hand side of the above equation corresponds to a $1$-coboundary in the $\rm \check{C}$ech cohomology. This observation prompts us to define a class that quantifies the obstruction to constructing a GH connection on a GH vector bundle.

\begin{Def}\label{Def:Cech-Atiyah}
The \textbf{Atiyah class} $[\alpha_E]\in  \check{\mathrm{H}}^1(M,G^*M\otimes \End(E))$
of a GH vector bundle $E$ on a regular GC manifold $(M,\mathcal{J})$ is given by the $\rm \check{C}$ech cocycle 
\begin{equation*}
\alpha_E=\{U_{ij},\varphi_{j}^{-1}\circ (\varphi_{ij}^{-1}d_+ (\varphi_{ij}))\circ \varphi_j\},
\end{equation*}
where $d_+$ is the Lie algebroid differential of the $+i$-eigenbundle  $L_+$ of $\mathcal{J}$.
\end{Def}

In summary, we have the following:
\begin{Thm}\cite{DGA2023}
A  GH vector bundle $E$ over a GC manifold $M$ admits a GH connection if and only if  $[\alpha_E]\in \check{\mathrm{H}}^1(M,G^*M\otimes \End(E))$ vanishes.
\end{Thm}

Next, we recall another way for defining the Atiyah class of a GH vector bundle. This approach stems from the definition of the Atiyah class for a Lie pair, which was introduced by Chen, Sti\'enon and Xu in their work \cite{CMP16}.

Consider the Lie pair $(T_{\mathbb{C}} M,\rho(L_{-}))$ and denote by $q\colon T_{\mathbb{C}} M\to T_{\mathbb{C}} M/\rho(L_{-})$ the natural projection.
Observe that \[(\rho(L_{-}))^\perp=L_{-}\cap T_{\mathbb{C}}^* M=G^*M=q^*(T_{\mathbb{C}} M/\rho(L_{-}))^*,\] 
where $(\rho(L_{-}))^\perp$ denotes the annihilator of $\rho(L_{-})$ of $T^*_{\mathbb{C}}M$.

Choose an $T_{\mathbb{C}} M$-connection $\nabla: \Gamma(E)\to \Gamma(T^*_{\mathbb{C}}M\otimes E)$ which extends the flat $\rho(L_-)$-connection  on $E$. The curvature of $\nabla$ induces a section $R^{\nabla}\in \Gamma\big(\rho(L_-)^*\otimes(\rho(L_{-}))^\perp\otimes \End(E)\big)$ or, equivalently, a bundle map $R^{\nabla}\colon \rho(L_{-})\otimes \big(T_{\mathbb{C}} M/\rho(L_{-})\big)\to \End(E)$  given by
\begin{equation}\label{Eqt：curvature}
R^{\nabla}\big(a,q(u)\big)={\nabla}_a\nabla_u -\nabla_u {\nabla}_a -\nabla_{[a,u]} ,\qquad \forall a\in \Gamma(\rho(L_-)),{u}\in \Gamma(T_{\mathbb{C}} M).
\end{equation}
This cohomology class \[[R^{\nabla}]\in {\mathrm{H}}^1_{\rm CE}(\rho(L_{-}),G^*M\otimes \End(E))\] is called the \textbf{Atiyah class} of the Lie pair $(T_{\mathbb{C}} M,\rho(L_{-}))$ with respect to the $\rho(L_{-})$-module $E$. 
It vanishes if and only if there is an $\rho(L_{-})$-compatible $T_{\mathbb{C}} M$-connection on $E$ (see \cite{CMP16, DGA2023} for more details).

\begin{Thm}\label{Thm:Cech-curvature}
The natural injection
$$\Pi_1\colon \check{\mathrm{H}}^1(\mathcal{U},G^*M\otimes \End(E)) \to \mathrm{H}^1_{\rm CE}(\rho(L_-), G^*M\otimes \End(E)) $$ maps the $\check{C}$ech cohomology class $[\alpha_E]$ to the Lie algebroid cohomology class $[R^\nabla]$. 
\end{Thm}
\begin{proof}
Suppose that  $(U_i;z,p,q)$ is a canonical coordinate on $U_i$ and $\varphi_i\colon E|_{U_i}\to U_i\times \C^r$ is a trivialization.  
Let $\nabla$ be an $T_\C M$-connection extending the $\rho(L_-)$-module structure on $E$ and $\{e_\alpha\}$  a basis of $\mathcal{O}(E,U_i).$
We define the Christoffel symbols by 
$$\nabla_{\frac{\partial}{\partial {z}_k}} e_\alpha:=\sum_\mu f_{k\alpha}^\mu e_\mu, \quad f_{k\alpha}^\mu\in C^\infty(U_i).$$
From Equation \eqref{Eqt：curvature}, we derive an explicit formula for the curvature:
$$R^\nabla|_{U_i}=\sum_{k,\alpha,\mu}(\dce f_{k\alpha}^\mu)\otimes dz_k\otimes e_\alpha^*\otimes e_\mu=\sum_{k,\alpha,\mu}\dce(f_{k\alpha}^\mu dz_k\otimes e_\alpha^*\otimes e_\mu),$$
where $\sum_{k,\alpha,\mu}f_{k\alpha}^\mu dz_k\otimes e_\alpha^*\otimes e_\mu\in\Gamma(G^*M\otimes \End(E), U_i ).$

Now, let $(U_j;z',p',q')$ be another canonical coordinate on $U_j$ and $\{t_\beta\}$ a basis of $\mathcal{O}(E,U_j).$ Assume that the Christoffel symbols in this coordinate system are given by 
$$\nabla_{\frac{\partial}{\partial {z'}_k}} t_\beta:=\sum_{\nu} g_{k\beta}^\nu t_\nu, \quad g_{k\beta}^\nu\in C^\infty(U_j).$$
Suppose that the local bases $\{ e_\alpha\}$ and $\{t_\beta\}$ of $\mathcal{O}(E,{U_{ij}})$ are related by the transition matrices
\begin{equation}\label{Eqt:transition}
e_\alpha=\sum_\beta t_\beta A_{\beta\alpha},\qquad  t_\beta=\sum_\alpha e_\alpha A^{\alpha\beta}
\end{equation}
where $(A_{\beta\alpha})^{-1}=(A^{\alpha\beta})$ and $A_{\beta\alpha}\in\mathcal{O}(U_{ij}).$ 
Using Equation \eqref{Eqt:transition}, we have
\begin{align*}
 \nabla_{\frac{\partial}{\partial {z}_k}} e_\alpha&= \sum_{\beta}\frac{\partial A_{\beta\alpha}}{\partial {z}_k} t_\beta+\sum_{\beta}
 A_{\beta\alpha} \nabla_{\frac{\partial}{\partial {z}_k}} t_\beta\\
 &=\sum_{\beta}\frac{\partial A_{\beta\alpha}}{\partial {z}_k} t_\beta+\sum_{\beta,l,\nu}
 A_{\beta\alpha}\frac{\partial {z'}_l}{\partial {z}_k}g_{l\beta}^\nu t_\nu\\
 &=\sum_{\beta,\mu}\frac{\partial A_{\beta\alpha}}{\partial {z}_k} 
 e_\mu A^{\mu \beta}+\sum_{\beta,l,\mu,\nu}
 A_{\beta\alpha}\frac{\partial {z'}_l}{\partial {z}_k}g_{l\beta}^\nu e_\mu A^{\mu \nu}.
\end{align*}
From this, we identify the Christoffel symbols in the $\{e_\alpha\}$ basis as:
\begin{equation}\label{Eqt:Christoffel-symbol}
f_{k\alpha}^\mu=\sum_{\beta}\frac{\partial A_{\beta\alpha}}{\partial {z}_k} 
 A^{\mu \beta}+\sum_{\beta,\nu,l}
 A_{\beta\alpha}\frac{\partial {z'}_l}{\partial {z}_k}g_{l\beta}^\nu  A^{\mu \nu}.
 \end{equation}

According to Proposition \ref{Prop:Cech-D} and Remark \ref{Rem:D-Cech}, it suffices to show that 
\begin{equation}\label{Eqt:curvature-cech}
\sum_{k,\beta,\nu} g_{k\beta}^\nu dz'_k\otimes t_\beta^*\otimes t_\nu-\sum_{k,\alpha,\mu} f_{k\alpha}^\mu dz_k\otimes e_\alpha^*\otimes e_\mu  =\varphi_{j}^{-1}\circ (\varphi_{ij}^{-1}d_+ (\varphi_{ij}))\circ \varphi_j\in \mathcal{O}( G^*M\otimes \End(E), U_{ij} ).    
\end{equation}

Regard the expression $\varphi^{-1}_j\circ (\varphi^{-1}_{ij}d_+(\varphi_{ij}))\circ\varphi_j$ as a homomorphism of local GH section spaces from $\mathcal{O}(E,{U_{ij}})$ to $\mathcal{O}(G^*M\otimes E,{U_{ij}})$, then for any $\sum_\beta h_\beta t_\beta\in\mathcal{O}(E,{U_{ij}})$ we have
\begin{align*}
&\qquad\varphi^{-1}_j\circ (\varphi^{-1}_{ij}d_+(\varphi_{ij}))\circ\varphi_j(\sum_\beta h_\beta t_\beta)\\
&=\varphi^{-1}_j\circ(\varphi^{-1}_{ij}\circ d_+\circ \varphi_{ij}-d_+)(\sum_\beta h_\beta \varphi_j(t_\beta))\\
&=\varphi^{-1}_i\circ d_+\circ\varphi_i(\sum_\beta h_\beta t_\beta)-\sum_\beta(d_+ h_\beta)\otimes t_\beta\\
&=\varphi^{-1}_i\circ d_+(\sum_{\alpha,\beta} h_\beta A^{\alpha\beta}\varphi_i(e_\alpha))-\sum_\beta(d_+ h_\beta)\otimes t_\beta\\
&=\sum_{\alpha,\beta} A^{\alpha\beta} (d_+ h_\beta)\otimes e_\alpha+\sum_{\alpha,\beta} h_\beta (d_+A^{\alpha\beta})\otimes e_\alpha-\sum_\beta(d_+ h_\beta)\otimes t_\beta\\
&=\sum_{\alpha,\beta,k}h_\beta A_{k \alpha}(d_+ A^{\alpha \beta}) \otimes t_k.
\end{align*}
Here $d_+(\varphi_{ij}):=d_+\circ \varphi_{ij}-\varphi_{ij}\circ d_+$.

On the other hand, using Equation \eqref{Eqt:Christoffel-symbol},  we have
\begin{align*}
&\quad \sum_{k,\beta,\nu}(g_{k\beta}^\nu dz'_k\otimes t_\beta^*\otimes t_\nu-\sum_{k,\alpha,\mu} f_{k\alpha}^\mu dz_k\otimes e_\alpha^*\otimes e_\mu  )(\sum_\beta h_\beta t_\beta)\\
&=\sum_{k,\beta,\nu}
g_{k\beta}^\nu h_\beta dz'_k\otimes  t_\nu-\sum_{k,\alpha,\mu,\lambda}f_{k\alpha}^\mu h_\lambda A^{\alpha\lambda} dz_k\otimes e_\mu \\
&=\sum_{k,\beta,\nu}
g_{k\beta}^\nu h_\beta dz'_k\otimes  t_\nu-\sum_{k,\alpha,\mu,\lambda}(\sum_{\beta}\frac{\partial A_{\beta\alpha}}{\partial {z}_k} 
 A^{\mu \beta}+\sum_{\beta,l,\nu}
 A_{\beta\alpha}\frac{\partial {z'}_l}{\partial {z}_k}g_{l\beta}^\nu  A^{\mu \nu}) h_\lambda A^{\alpha\lambda} dz_k\otimes e_\mu\\
&=-\sum_{k,\alpha,\beta,\mu,\lambda}\frac{\partial A_{\beta\alpha}}{\partial {z}_k} 
 A^{\mu \beta} h_\lambda A^{\alpha\lambda} dz_k\otimes e_\mu\\
&=\sum_{\alpha,\beta,\lambda} A_{\beta\alpha}h_\lambda d_+(A^{\alpha\lambda})\otimes t_\beta.
\end{align*}
Therefore, Equation \eqref{Eqt:curvature-cech}
holds. 
\end{proof}

In conclusion, we find that $[R^\nabla]$ vanishes if and only if $E$ possesses a GH connection.


\section{Correspondence of  extension classes and $\rm{\check{C}}$ech classes}
In this section, we demonstrate that the Atiyah classes, which are defined in terms of extension classes and $\rm{\check{C}}$ech classes, are identical, as evidenced by the isomorphism between their respective underlying groups.

Let us first recall the construction of the jet bundle $\mathfrak{J}^1 E$ of a GH vector bundle $E$ over a regular GC manifold $M$. 
Denote by $\mathcal{O}_m(E)$ the space of local GH sections around $m\in M$. Two local sections $e, t\in \mathcal{O}_m(E)$ are said to be equivalent if 
$$e(m)=t(m),\qquad e_{*,m}=t_{*,m}.$$
We denote the equivalence class of $e$ at $m$ as $[e]_m$, which is called the (first) jet of $e$ at $m$. Define
$$\mathfrak{J}^1 E=\{[e]_m| m\in M, e\in  \mathcal{O}_m(E) \},$$
which is a GH vector bundle and fits into a short exact sequence of GH vector bundles 
\begin{equation}\label{seq:GHVB-jet}
 0\longrightarrow  G^*M\otimes E\stackrel{I} \longrightarrow \mathfrak{J}^1 E\stackrel{\pi^1}\longrightarrow E\longrightarrow 0,
\end{equation}
where $G^*M$ is the GH cotangent bundle of $M$ \cite{DGA2023}. 
Now we recall the third definition of the Atiyah class of a GH vector bundle, which is characterized as the extension class associated with the jet bundle.
\begin{Def}\label{Def:Atiyah-jet}
The Atiyah class of $E$ is defined to be the first extension class 
\begin{equation*}\label{Def:Atiyahclass-jet}
A(E)\in \Ext^1_{\mathcal{O}_M}(E,G^*M\otimes E)
\end{equation*}
of the short exact sequence \eqref{seq:GHVB-jet}, where $\Ext^1_{\mathcal{O}_M}(E,G^*M\otimes E)$ denotes the group of isomorphic classes of extensions  of  $G^*M\otimes E$ by $E$ in the category of GH vector bundles.
\end{Def}

In general, there is no canonical choice of splitting for the short exact sequence \eqref{seq:GHVB-jet}. However, the induced short exact sequence 
\begin{equation*}\label{seq:GHVB-jet-section}
 0\longrightarrow  \mathcal{O}(G^*M\otimes E)\stackrel{I} \longrightarrow \mathcal{O}(\mathfrak{J}^1 E)\stackrel{\pi^1}\longrightarrow \mathcal{O}(E)\longrightarrow 0,
\end{equation*}
at the level of spaces of GH sections splits canonically: if $e$ is a GH section of $E$, then $\sigma(e):=[e]$ is a GH section of $\mathfrak{J}^1 E$ such that $\pi^1(\sigma(e))=e$.  Note that the splitting $\sigma$ is not $\mathcal{O}_M$-linear. In fact, in a local coordinate system $(U,\varphi;z,p,q)$ of $M$, we have
\begin{equation*}
\sigma(fe)=f\sigma(e)+\sum_{\lambda} \frac{\partial f}{\partial z_\lambda}I(dz_\lambda\otimes e),
\end{equation*}
where $f$ is a local GH function on $U$ and $e\in\mathcal{O}(E,U)$. Hence, a basis $\{e_\alpha\}$   of $\mathcal{O}(E,U)$ gives rise to a basis $\{\sigma(e_\alpha), I(dz_\lambda\otimes e_\alpha)\}$ of
$\mathcal{O}(\mathfrak{J}^1 E,U)$. Let $V$ be another open subset of $M$ with coordinate $(z',p',q')$ satisfying $U\cap V\neq \emptyset$. Assume that  $\{t_\beta\}$ is a basis of $\mathcal{O}(E,V)$ and $e_\alpha=\sum_{\beta}A_{\beta\alpha}t_\beta$ on $U\cap V$, then the transition functions of $\mathfrak{J}^1 E$ are given by
\begin{equation}\label{Eqt:JE-transition-sigma0}
I(dz_\lambda\otimes e_\alpha)=\sum_{\beta,\mu} A_{\beta\alpha}\frac{\partial z_\lambda}{\lambda z'_\mu}I(dz'_\mu\otimes t_\beta),
\end{equation}
and 
\begin{equation}\label{Eqt:JE-transition-sigma}
\sigma(e_\alpha)=\sum_{\beta} A_{\beta\alpha}\sigma(t_\beta)+\sum_{\beta,\mu} \frac{\partial A_{\beta\alpha}}{\lambda z'_\mu}I(dz'_\mu\otimes t_\beta).
\end{equation}

In generalized geometry, we now establish an Extension-to-$\rm{\check{C}}$ech correspondence analogous to that found in complex geometry, formalized as follows:
\begin{Thm}\label{Thm:EXT-Cech}
Let $E_1$ and $E_2$ be two GH vector bundles over a GC manifold $M$.
There exists a natural isomorphism
$$\Phi\colon \Ext^1_{\mathcal{O}_M}(E_2,E_1)\to\check{\mathrm{H}}^1(\mathcal{U}, E_2^*\otimes E_1),$$
where $\mathcal{U}$ is an open cover of $M$.
\end{Thm}

\begin{proof}
Consider a short exact sequence representing an extension of $E_2$ by $E_1$:
\begin{equation*}\label{seq:AtiyahDBESplit}
	\xymatrix@C=2pc@R=2pc{
		0\ar[r]&E_1\ar[r]^{\iota}
		&E\ar[r]^{P}
		&E_2\ar[r]
		&0.
	}
\end{equation*}
By the local splittability of such extensions, there exists an open cover $\mathcal{U}=\{U_i \}_{i\in \Lambda}$ of $M$ such that the sequence splits over each $U_i.$ This means that there exist a short exact sequnce
\begin{equation*}\label{seq:AtiyahDBESplit}
	\xymatrix@C=2pc@R=2pc{
		0&E_1|_{U_i}\ar[l]
		&E|_{U_i}\ar[l]_{\xi_i}
		&E_2|_{U_i}\ar[l]_{\eta_i}
		&0,\ar[l] 
	}
\end{equation*}
satisfying $P|_{U_i}\circ \eta_i=\id_{E_2|_{U_i}}$, $\xi_i\circ \iota|_{U_i}=\id_{E_1|_{U_i}}$ and $\iota|_{U_i}\circ\xi_i+\eta_i\circ P|_{U_i}=\id_{E|_{U_i}}.$

On non-empty intersection $U_{ij}=U_i \cap U_j$, the two sets of splitting maps $(\xi_i,\eta_i)$ and $(\xi_j,\eta_j)$ satisfy a key compatibility relation:
$$\xi_i\circ\eta_j=-\xi_j\circ \eta_i\colon E_2|_{U_{ij}}\to E_1|_{U_{ij}}.$$
This defines a 1-cochain $T=(\xi_i\circ\eta_j)$  in the $\rm{\check{C}}$ech space $C^1(\mathcal{U},E_2^*\otimes E_1)$ associated to the locally free sheaf $E_2^*\otimes E_1$. To verify $T$ is a cocycle, note that the splitting maps satisfy $\eta_i-\eta_j=\iota\circ \xi_j\circ\eta_i$, and a direct computation shows the $\rm{\check{C}}$ech differential vanishes:
$$(dT)_{ijk}=\xi_j\circ\eta_k-\xi_i\circ\eta_k+\xi_i\circ\eta_j=0.$$ Hence we define the map $\Phi$
 by assigning to each extension class its corresponding cocycle class:
$$\Phi\colon\Ext^1_{\mathcal{O}_M}(E_2,E_1)\to \check{\mathrm{H}}^1(\mathcal{U}, E_2^*\otimes E_1),\qquad [E]\to [(\xi_i\circ\eta_j)] .$$

First we check that $\Phi$ is well-defined. It suffices to prove that equivalent extensions produce the same cohomology class.
Let $\theta\colon E\to E'$ be an isomorphism of GH vector bundles between two extensions:
\begin{equation*}
	\xymatrix@C=2pc@R=1pc{
		0\ar[r]&E_1\ar@{=}[d]\ar[r]^{\iota}
		& E\ar[r]^{P}\ar[d]^{\theta}
		&E_2\ar[r]\ar@{=}[d]
		&0\\
		0\ar[r]&E_1\ar[r]^{\iota'}
		&E'\ar[r]^{P'}
		&E_2\ar[r]
		&0.
	}
\end{equation*}
Let $(\xi'_i,\eta'_i)$ denote GH splitting maps for $E'$ over $U_i$. 
Then the isomorphism $\theta$ determines a 0-cochain in the $\rm{\check{C}}$ech complex:
$$h=(h_i:=\xi'_i\circ\theta \circ\eta_i-\xi_i \circ\theta^{-1} \circ\eta'_i)\in C^0(\mathcal{U},E_2^*\otimes E_1).$$
On the intersection $U_{ij}$, we have $\xi_i\circ\eta_j-\xi'_i\circ\eta'_j =(dh)_{ij}$. This implies that the two cocycles belong to the same cohomology class. 
Moreover, an analogous argument shows that the cohomology class $[(\xi_i\circ\eta_j)]$ is independent of the specific choice of local splittings $(\xi_i,\eta_i)$ for the original extension $E.$ Together, these observations ensure that $\Phi$ is well-defined.

Next, we verify the injectivity of 
$\Phi$. Suppose 
$E$ and $E'$ are extensions such that their cocycles satisfy
\begin{equation}\label{Eqt:hcocycle}
\xi_i\circ\eta_j-\xi'_i\circ\eta'_j =(dh)_{ij},
\end{equation}
for some 0-cochain $h=(h_i)\in C^0(\mathcal{U},E_2^*\otimes E_1)$. We construct an isomorphism between 
$E$ and $E'$. Locally on $U_i$, define a morphism of GH vector bundles:
$$\theta_i:=\iota'\circ h_i\circ P+\eta'_i\circ P+\iota'\circ\xi_i.$$
One readily verifies that $\theta_i$ is invertible, with inverse given by
$$\iota\circ\xi'_i+\eta_i\circ P'-\iota\circ h_i\circ P'.$$
By the cocycle condition \eqref{Eqt:hcocycle}, we find $\theta_i=\theta_j$ on $U_{ij}$, so 
$\{\theta_i\}$ glues to a global isomorphism 
$E\to E'$. Thus $\Phi$ is injective.

Finally, we show that $\phi$ is surjective. Let $(T_{\alpha \beta})\in C^1(\mathcal{U},E_2^*\otimes E_1)$ be a $\rm{\check{C}}$ech cocycle. Denote by $\varphi^{E_1}_{\alpha \beta}$ and $\varphi^{E_2}_{\alpha \beta}$ the transition matrices of the GH vector bundles $E_1$ and $E_2$  with respect to $\mathcal{U}$. Consider the smooth vector bundle $E_1\oplus E_2$; by Proposition 3.2. in \cite{DGA2023}, the transition matrices
$$\psi_{\alpha \beta}:=
\begin{pmatrix}
\varphi^{E_1}_{\alpha \beta} &  \varphi^{E_1}_{\beta}\circ T_{\alpha \beta}\circ (\varphi^{E_2}_{\alpha})^{-1}\\
0 & \varphi^{E_2}_{\alpha \beta}
\end{pmatrix}
$$
define a GH structure on $E_1\oplus E_2$. This yields a short exact sequence
\begin{equation*}
\xymatrix@C=2pc@R=2pc{
0\ar[r]&E_1\ar[r]^{\iota}
&E_1 \oplus E_2\ar[r]^{P}
&E_2\ar[r]
&0,
}
\end{equation*}
(where $\iota(e)=(e,0)$ and $P(e_1,e_2)=e_2$) whose image under $\Phi$ is precisely the cohomology class of $(T_{\alpha \beta})$. Thus, $\Phi$ is surjective.

In summary, $\Phi$ is an isomorphism.
\end{proof}

Let $E$ be a GH vector bundle over a regular GC manifold $M$. We are now ready to state the main result in this section.

\begin{Thm}\label{Thm:2main}
Given an open covering $\mathcal{U}=\{U_i\}_{i\in\Lambda}$ of $M$, the natural isomorphism
$$\Phi\colon\Ext^1_{\mathcal{O}_M}(E,G^*M\otimes E)\to \check{\mathrm{H}}^1(\mathcal{U}, G^*M\otimes \End (E))$$
maps the extension class $A(E)$ of the short exact sequence \eqref{seq:GHVB-jet} to the $\rm{\check{C}}$ech class $[\alpha_E]$.
\end{Thm}

\begin{proof}
Let $\varphi_i\colon E|_{U_i}\to U_i\times \C^r$ be a local trivialization of $E$ and $\{e_\alpha\}$  the basis of $\mathcal{O}(E,U_i)$  induced by $\varphi_i$. Let
$(z,p,q)$ denote a canonical coordinate system on $U_i$. We introduce a local splitting of the sequence \eqref{seq:GHVB-jet} over each $U_i$:
\begin{equation*}\label{seq:localSplit}
	\xymatrix@C=2pc@R=2pc{
		0&G^*M\otimes E|_{U_i}\ar[l]
		&\mathfrak{J}^1E|_{U_i}\ar[l]_{~~~\xi_i}
		&E|_{U_i}\ar[l]_{\eta_i}
		&0,\ar[l] 
	}
\end{equation*}
where the splitting maps $\xi_i$ and $\eta_i$ are defined as follows: 
$$\eta_i(\sum_\alpha f_\alpha e_\alpha)=\sum_\alpha f_\alpha\sigma( e_\alpha),\quad f_\alpha\in\mathcal{O}_M(U_i),$$
and $$\xi_i\big(\sum_\alpha f_\alpha\sigma( e_\alpha)+\sum_{\alpha,\lambda}f_{\lambda\alpha} I(dz_\lambda\otimes e_\alpha)\big)=\sum_{\alpha,\lambda}f_{\lambda\alpha} dz_\lambda\otimes e_\alpha , \quad f_\alpha, f_{\lambda\alpha}\in\mathcal{O}_M(U_i).$$ 

On another open set $U_j$ with canonical coordinates $(z',p',q')$, let $\varphi_j\colon E|_{U_j}\to U_j\times \C^r$ be a local trivialization, and $\{t_\beta\}$ the basis of  $\mathcal{O}(E,U_j)$ induced by $\varphi_j$. Suppose that the local bases $\{ e_\alpha\}$ and $\{t_\beta\}$ of $\mathcal{O}(E,{U_{ij}})$ are related by the transition matrices
\begin{equation}\label{Eqt:transition2}
e_\alpha=\sum_\beta t_\beta A_{\beta\alpha},\qquad  t_\beta=\sum_\alpha e_\alpha A^{\alpha\beta}
\end{equation}
where $(A^{\alpha\beta})$ is the  inverse matrix of $(A_{\beta\alpha})$ and $A_{\beta\alpha}\in\mathcal{O}(U_{ij}).$ 


Recall from Definition \ref{Def:Cech-Atiyah} that the $\rm{\check{C}}$ech cocycle representing the Atiyah class 
is given by:
$$\alpha_E=\{U_{ij}, \varphi^{-1}_j\circ (\varphi^{-1}_{ij}d_+(\varphi_{ij}))\circ\varphi_j\}.$$
By Theorem \ref{Thm:EXT-Cech}, the map 
$\Phi$ assigns to the extension class 
associated to $\mathfrak{J}^1E$ the $\rm{\check{C}}$ech cocycle $\{U_{ij}, \xi_i\circ \eta_j \}$. We now verify that these cocycles coincide:
$$\xi_i\circ \eta_j=\varphi^{-1}_j\circ (\varphi^{-1}_{ij}d_+(\varphi_{ij}))\circ\varphi_j.$$

Regard $\varphi^{-1}_j\circ (\varphi^{-1}_{ij}d_+(\varphi_{ij}))\circ\varphi_j$ as a homomorphism of local GH section
spaces from $\mathcal{O}(E,{U_{ij}})$ to $\mathcal{O}(G^*M\otimes E,{U_{ij}})$. 
By Theorem \ref{Thm:Cech-curvature}, for any $\sum_\beta h_\beta t_\beta\in\mathcal{O}(E,{U_{ij}})$ we have
$$\varphi^{-1}_j\circ (\varphi^{-1}_{ij}d_+(\varphi_{ij}))\circ\varphi_j(\sum_\beta h_\beta t_\beta)=\sum_{\alpha,\beta,k}h_\beta A_{k \alpha}(d_+ A^{\alpha \beta}) \otimes t_k.$$ 
On the other hand, using Equations \eqref{Eqt:JE-transition-sigma} and \eqref{Eqt:transition2}, we have 
\begin{align*}
 \xi_i\circ \eta_j(\sum_\beta h_\beta t_\beta)
&=\xi_i\big(\sum_\beta h_\beta\sigma( t_\beta)\big)\\
&=\xi_i\big(\sum_{\alpha,\beta} h_\beta A^{\alpha\beta}\sigma(e_\alpha)+\sum_{\alpha,\beta,\lambda}h_\beta\frac{\partial A^{\alpha\beta}}{\partial z_\lambda} I(dz_\lambda\otimes e_\alpha)\big)\\
&=\sum_{\alpha,\beta,\lambda}h_\beta\frac{\partial A^{\alpha\beta}}{\partial z_\lambda} dz_\lambda\otimes e_\alpha\\
&=\sum_{\alpha,\beta,k}h_\beta A_{k \alpha}(d_+ A^{\alpha \beta}) \otimes t_k.
\end{align*}
Thus $\xi_i\circ \eta_j=\varphi^{-1}_j\circ (\varphi^{-1}_{ij}d_+(\varphi_{ij}))\circ\varphi_j$ on $U_{ij}$, so  the natural isomorphism $\Phi$ maps the extension class $A(E)$ to the $\rm{\check{C}}$ech class $[\alpha_E]$.
\end{proof}

Applying Theorem \ref{Thm:2main}, we recover a result of Lang-Jia-Liu \cite{DGA2023}:
\begin{Cor}
Let $E$ be a GH vector bundle over a regular GC manifold $M$. Then $E$ admits a GH connection if and only if the short exact sequence \eqref{seq:GHVB-jet} splits.
\end{Cor}

\section{From extension classes to Lie algebroid cohomology classes}
In this section, we demonstrate that the Atiyah classes, defined through the extension class of the first jet bundle $\mathfrak{J}^1E$ and the cohomology class of the Lie pair $(T_{\mathbb{C}} M,\rho(L_{-}))$, are canonically identical. To establish this correspondence, we employ the concept of the jet bundle associated with a Lie pair, which was introduced by Chen, Stiénon and Xu in \cite{CMP16}.
\subsection{The jet bundle arising from a Lie pair}\label{Sec:jetbundle}

We start by recalling the jet bundle arising from a Lie pair. To be best of our knowledge, this construction first appeared explicitly in \cite{CMP16}.
\begin{Def}\label{Def:jetLiepair} Let  $(L,A)$ be a complex Lie pair over $M$ and $E$ an $A$-module.
	An $L$-jet (of order $1$) on $E$ extending the $A$-action $\nabla$ is a pair $(D_x,e_x)$ consisting of a $\C$-linear map $D_x\colon \bsection{{E^*}}\to  L^* _x$ and a point $e_x$ in the fiber of $E$ over $x\in M$, satisfying
	\begin{eqnarray*}
		\label{Eqt:Ljet1}
		\langle D_x(\varepsilon),a_x \rangle&=&
		\langle \nabla_{a_x}\varepsilon,e_x \rangle  ,\\
		\label{Eqt:Ljet2}
		D_x(f\varepsilon) &=&f(x)\cdot D_x(\varepsilon)+\langle \varepsilon_x,e_x \rangle\cdot\rho_{_L}^*(df),
	\end{eqnarray*}
	for all $a_x\in A_x$, $\varepsilon\in\bsection{{E^*}}$, and $f\in C^\infty(M,\C)$.
\end{Def}
The jet bundle $\mathscr{J}^1_{L/A}E$ is the manifold whose points are $L$-jets on $E$. It is a vector bundle over $M$: the projection $\mathscr{J}^1_{L/A}E\to M$ maps $(D_x,e_x)$ to $x$. 
A global section of $\mathscr{J}^1_{L/A}(E)$ is of the form $(D,e)$ consisting of   an $\C$-linear map $D:~\Gamma({E^*})\to \Gamma( L^* )$ and a section $e\in \bsection{E} $ satisfying 
\begin{eqnarray*}
	\langle D(\varepsilon), a\rangle&=& \langle \nabla_a \varepsilon,e\rangle ,\\
	D(f\varepsilon) &=& fD(\varepsilon)+\langle \varepsilon,e\rangle \cdot \rho_{_L}^*(df),
\end{eqnarray*}
for all $f\in C^\infty(M, \C)$, $\varepsilon\in\Gamma({E^*})$, and $a\in\Gamma(A)$.
Moreover, it fits into a short exact sequence of vector bundles over $M$:
\begin{equation}\label{Dia:jetbundle}
	0\longrightarrow  {A^\perp}\otimes E\stackrel{\kappa} \longrightarrow \mathscr{J}^1_{L/A}(E)\stackrel{\gamma}\longrightarrow E\longrightarrow 0,
\end{equation}
where $A^\perp$ denotes the annihilator of $A$ of $L^*.$

Let $B$ denote the quotient bundle $L/A$ and $q\colon L\to B$ the natural projection.
A splitting $j\colon B \to L$ naturally determines a splitting $\zeta\colon \Gamma(E)\to \Gamma(\mathscr{J}^1_{L/A}(E))$ of the short exact sequence of spaces of smooth sections 
\begin{equation*}
	0\longrightarrow  \Gamma({A^\perp}\otimes E)\stackrel{\kappa} \longrightarrow \Gamma(\mathscr{J}^1_{L/A}(E))\stackrel{\gamma}\longrightarrow \Gamma(E)\longrightarrow 0,
\end{equation*}
induced by the short exact sequence \eqref{Dia:jetbundle}. For every $e\in\Gamma(E)$, denote by $\zeta(e):=(D^{\zeta(e)}, e)$. The $\C$-linear map $D^{\zeta(e)}$ determined by the section $\zeta(e)$ satisfies
\begin{eqnarray*}
	\langle D^{\zeta(e)}(\varepsilon), a\rangle&=& \langle \nabla_a \varepsilon,e\rangle ,\\
	\langle D^{\zeta(e)}(\varepsilon), j(b)\rangle &=& \rho(j(b))\langle \varepsilon,e\rangle ,
\end{eqnarray*}
for all $e\in\Gamma(E)$, $\varepsilon\in\Gamma(E^*)$, $a\in\Gamma(A)$
and $b\in\Gamma(B)$.

The proof of the following lemma is a tedious computation, which we omit.
\begin{lem}\label{Lem:zeta-Liepair}
The splitting $\zeta\colon \Gamma(E)\to \Gamma(\mathscr{J}^1_{L/A}(E))$ is not $C^\infty(M, \C)$-linear. For every $e\in\Gamma(E)$ and $f\in C^\infty(M, \C)$, we have 
\begin{equation*}
\zeta(fe)=f\zeta(e)+\kappa\big((q^*j^*\rho_L^*(df))\otimes e\big).
\end{equation*}
\end{lem}

The jet bundle $\mathscr{J}^1_{L/A}(E)$ can be naturally endowed with an $A$-action.

\begin{prop}\cite{CMP16}
	\begin{itemize}
		\item[(1)] 
		The jet   vector bundle $\mathscr{J}^1_{L/A}(E)$ is a representation over $A$ --- The covariant derivative 
		\[\Gamma(A)\times\Gamma\big(\mathscr{J}^1_{L/A}(E)\big)\to\Gamma\big(\mathscr{J}^1_{L/A}(E)\big)\]
		is given by 
		$$\nabla_a(D,e):=(\nabla_a D,\nabla_ae),$$
		where the $\K$-linear map $\nabla_a D\colon \bsection{{E^*}}\to\bsection{ L^* }$ is defined as follows:
		$$\langle (\nabla_a D)(\varepsilon), l \rangle
		:=\rho_{_A}(a)\langle D(\varepsilon),l \rangle
		-\langle D(\nabla_a \varepsilon),l \rangle -\langle D(\varepsilon),[a,l]_{_L} \rangle,\quad \forall l\in \bsection{L}.$$
		\item[(2)] Sequence~\eqref{Dia:jetbundle} is a short exact sequence of $A$-representations.
	\end{itemize}
\end{prop}

\subsection{Equivalence of extension classes and cohomology classes}
In this section, we prove that the extension class and Lie algebroid cohomology class are one and the same. Throughout this section, we assume that $E$ is a GH vector bundle over a regular GC manifold $M$.

Let  $\mathscr{J}^1 E$ denote the first jet bundle derived from the Lie pair $(T_\C M,\rho(L_{-}))$ as elaborated in Section \ref{Sec:jetbundle}. Then it fits into a short exact sequence of $\rho(L_{-})$-modules over $M$:
\begin{equation}\label{Dia:jetbundleGC}
	0\longrightarrow  {G^*M\otimes E}\stackrel{\kappa} \longrightarrow \mathscr{J}^1(E)\stackrel{\gamma}\longrightarrow E\longrightarrow 0,
\end{equation}
where $G^*M\otimes E$ is indeed a GH vector bundle. In the following, we will show that $\mathscr{J}^1(E)$ is also a GH vector bundle.

By choosing a splitting $j\colon T_\C M/\rho(L_{-}) \to T_\C M$, we subsequently obtain a splitting $\zeta$ for Sequence \eqref{Dia:jetbundleGC}.

\begin{lem}\label{lem:zetad-plus}
For every $e\in\mathcal{O}(E,U)$ and $f\in \mathcal{O}_M(U)$,  the following equality holds: 
\begin{equation*}
\zeta(fe)=f\zeta(e)+\kappa((d_{+}f)\otimes e).
\end{equation*}
\end{lem}
\begin{proof}
Since $f$ is a GH function, we know that  $df=d_{+}f\in \mathcal{O}(G^*M, U).$
By Lemma \ref{Lem:zeta-Liepair}, we have 
$$\zeta(fe)=f\zeta(e)+\kappa(s\otimes e),$$
where the section $s$ is defined by the formula
$\langle s, l\rangle=\langle df, jq(l)\rangle=\langle d_{+}f, l\rangle,$ for all $l\in\Gamma(T_\C M)$. From this definition, it immediately follows that $s=d_{+}f.$
\end{proof}

To establish the correspondence between extension classes and Lie algebroid cohomology classes, a clear elaboration of the relationship between $\mathscr{J}^1 E$ and $\mathfrak{J}^1 E$ is imperative. 
\begin{prop}
Let $M$ be a regular GC manifold. Then
$\mathscr{J}^1E$ is a GH vector bundle. 
\end{prop}

\begin{proof}
According to Definition \ref{Def:jetLiepair}, a local smooth section of $\mathscr{J}^1(E)$ on a open subset $U\subset M$ is of the form $(D,e)$ consisting of   an $\C$-linear map $D\colon\Gamma({E^*})\to \Gamma(T^*_\C M,U)$ and a local smooth section $e\in \Gamma(E, U) $ satisfying 
\begin{eqnarray*}
	\langle D(\varepsilon), a\rangle&=& \langle \nabla_a \varepsilon,e\rangle ,\\
	D(f\varepsilon) &=& fD(\varepsilon)+\langle \varepsilon,e\rangle \cdot df,
\end{eqnarray*}
for all $f\in C^\infty(U, \C)$, $\varepsilon\in\Gamma({E^*})$, and $a\in\Gamma(\rho(L_-), U)$.

Let $\{e_\alpha\}$ be a local GH basis of $\mathcal{O}(E,U)$, then any $e\in \Gamma(E, U)$ can be expressed as $\sum_{\alpha}f_\alpha e_\alpha$, where $f_\alpha\in C^\infty(U, \C)$. Given a local coordinate system  $(z,p,q)$ on $U$, the set $\{d{z}_\lambda\otimes e_\alpha\}$ forms a basis of $\mathcal{O}(G^*M\otimes E, U)$. Consequently, any local section $(D,e)$ admits a decomposition of the form
$$\sum_\alpha f_\alpha \zeta(e_\alpha)+\sum_{\lambda,\alpha}g_{\lambda\alpha} \kappa(d{z}_\lambda\otimes e_\alpha),$$ 
for some $g_{\lambda\alpha}\in C^\infty(U, \C)$.   Thus, the set $\{\zeta(e_\alpha),\kappa(d{z}_\lambda\otimes e_\alpha)\}$ constitutes  a basis for $\Gamma(\mathscr{J}^1E,U)$.

Let $V$ be another open subset of $M$ with coordinate $(z',p',q')$ satisfying $U\cap V\neq \emptyset$. Assume that  $\{t_\beta\}$ is a basis of $\mathcal{O}(E,V)$, then  $\{\zeta(t_\beta),\kappa(d{z}'_\mu\otimes t_\beta)\}$ forms a basis of $\Gamma(\mathscr{J}^1E,V)$. Assume that $e_\alpha=\sum_\beta A_{\beta \alpha}t_\beta$, where $A_{\beta\alpha}\in\mathcal{O}_M(U\cap V)$. Then we have
\begin{equation*}
\kappa(d{z}_\lambda\otimes e_\alpha)=\kappa(\sum_{\beta,\mu}\frac{\partial {z}_\lambda}{\partial {z'}_\mu}d{z'}_\mu\otimes A_{\beta \alpha}t_\beta)=\sum_{\beta,\mu} A_{\beta \alpha}\frac{\partial {z}_\lambda}{\partial {z'}_\mu}\kappa(d{z'}_\mu\otimes t_\beta).
\end{equation*}
By \cite[Example 2.8]{DGA2023}, we have $d_+A_{\beta\alpha}=\sum_{\mu}\frac{\partial A_{\beta \alpha}}{\partial {z'}_\mu}d{z'}_\mu.$
It follows from Lemma \ref{lem:zetad-plus} that
\begin{equation*}
\zeta(e_\alpha)=\sum_{\beta} A_{\beta \alpha}\zeta(t_\beta)+\sum_{\beta}\kappa(d_+A_{\beta \alpha}\otimes t_\beta)=\sum_{\beta} A_{\beta \alpha}\zeta(t_\beta)+\sum_{\beta,\mu}\frac{\partial A_{\beta \alpha}}{\partial {z'}_\mu}\kappa(d{z'}_\mu\otimes t_\beta).
\end{equation*}

We observe that the transition functions are GH functions. By Proposition \ref{prop:Explicit}, this establishes that $\mathscr{J}^1E$ is a GH vector bundle over $M$.
\end{proof}

Furthermore, we present the following result:
\begin{prop}\label{jetisomorphism}
The jet bundle $\mathscr{J}^1E$ is isomorphic to $\mathfrak{J}^1E$ as GH vector bundles.
\end{prop}

\begin{proof}
It suffices to define an isomorphism $\Pi$ of GH vector bundles such that the following diagram commutes in the category of GH vector bundles:
\begin{equation*}\label{seq:AtiyahDBESplit}
	\xymatrix@C=2pc@R=1pc{
		0\ar[r]&G^*M\otimes E\ar[r]\ar[d]^{\id}\ar[r]^{I}
		&\mathfrak{J}^1 E\ar[r]^{\pi^1}\ar[d]^{\Pi}
		&E\ar[r]\ar[d]^{\id}
		&0\\
		0\ar[r]&G^*M\otimes E\ar[r]^{\kappa}
		&\mathscr{J}^1 E\ar[r]^{\gamma}
		&E\ar[r]
		&0.
	}
\end{equation*}

For any local GH section $\phi$ of $E$ defined in a neighborhood of a point $x\in M$, we define a bundle map $\Pi_x\colon \mathfrak{J}^1_x E\to \mathscr{J}^1_x E$ as follows: $\Pi_x([\phi]_x)=(D^\phi,\phi_x)$, where the $\C$-linear map $D^\phi\colon \Gamma(E^*)\to (T_\C^* M)_x$  is specified by the following two equations:
\begin{equation}\label{Eqt:jet-def1}\langle D^\phi(\varepsilon), a\rangle:=\langle \nabla_{a}\varepsilon, \phi_x \rangle ,\end{equation}
and
\begin{equation}\label{Eqt:jet-def2}
\langle D^\phi(\varepsilon), j(b)\rangle:=j(b) \langle \varepsilon, \phi \rangle ,
\end{equation}
for any $a\in\rho(L_{-})_x$, $b\in (T_\C M/\rho(L_-))_x$ and $\varepsilon\in \Gamma(E^*)$. It is obvious that $(D^\phi,\phi_x)$ satisfies the conditions outlined in Definition \ref{Def:jetLiepair}. 

According to Proposition 2.22 in \cite{BCSX}, we obtain the relation $\langle D^\phi(\varepsilon), j(b)\rangle=\phi_{*,x}(j(b))(f_\varepsilon)$, where $f_\varepsilon$ represents the fiberwise linear function on E that is uniquely determined by $\varepsilon\in \Gamma(E^*)$. Consequently, Equations \eqref{Eqt:jet-def1} and \eqref{Eqt:jet-def2} are independent of the specific choices made for $\phi$, ensuring that $\Pi_x$  is well-defined. Furthermore, from their local expressions, it follows that $\Pi$ constitutes a homomorphism between GH vector bundles.

Next, we proceed to demonstrate that $\Pi_x$ is a bijection. If $\Pi_x([\phi]_x)=0$, it follows that $\phi_x=0$ and $\phi_{*,x}=0$, which implies that $[\phi]_x=0$. Therefore, $\Pi_x$ is injective. 
On the other hand, since ${\rm dim}\mathfrak{J}^1_x E={\rm dim} \mathscr{J}^1_x E$, we conclude that $\Pi_x$ is also surjective. Consequently, $\Pi_x$ is a bijection. 
Thus, $\Pi$ is an isomorphism of GH vector bundles.


\end{proof}

\begin{Rem}
Comparing with transition functions of the generalized holomorphic vector bundle $\mathfrak{J}^1 E$ as described in Equations \eqref{Eqt:JE-transition-sigma0} and \eqref{Eqt:JE-transition-sigma}, both bundles $\mathscr{J}^1E$ and $\mathfrak{J}^1 E$ are associated to the same cocycle. Hence, from the point of view of transtion functions, they are also isomorphic.
\end{Rem}

By Theorem \ref{Thm:NN}, a GH vector bundle is indeed an $L_-$-module and thus a $\rho({L_-})$-module. 
Let 
\begin{equation}\label{Seq:GHVB}
	0\longrightarrow  E_1 \longrightarrow E\longrightarrow E_2\longrightarrow 0
\end{equation}
be a short exact sequence of GH vector bundles over a regular GC manifold $M$. It can be also regarded as a short exact sequence of $\rho({L_-})$-modules. Denote by  $\mathcal{A}$  the category of left modules over universal enveloping algebra $\mathcal{U}(\rho(L_{-}))$ \cites{Quantum-groupoids, CMP16} of the Lie algebroid $\rho(L_{-})$. In particular, an $\rho(L_{-})$-module over $M$  is an object in $\mathcal{A}$. 
Denote by ${\Ext}^1_{\mathcal{A}}(E_2,E_1)$  the set of isomorphic classes of extensions  of  $E_2$ by $E_1$ in the category $\mathcal{A}$.

As a consequence of Theorem \ref{Thm:NN}, we have the following lemma.
\begin{lem}\label{Lem:extmap}
Let $E_1$ and $E_2$ be two GH vector bundles, then there exists a natural injection   
$\mu\colon {\Ext}^1_{\mathcal{O}_M}(E_2,E_1) \to {\Ext}^1_{\mathcal{A}}(E_2,E_1)$.
\end{lem}
\begin{proof}
It is obvious that $\mu$ is well-defined. It suffices to check that $\mu$ is injective.

Assume that $E$ and $E'$ are two GH vector bundle extension of $E_2$ by $E_1$, and $\phi\colon E\to E'$ is  an $\rho({L_-})$-module isomorphism. We aim to demonstrate that $\phi$ is also an isomorphism of GH vector bundles. Consider an open subset $U\subset M$, and assume that $\phi(e_i)=\sum_j f_{ij}t_j$, where $f_{ij}\in C^\infty(U, \C)$ and  $\{e_i\}$ and $\{t_j\}$ form  bases of $\mathcal{O}(E,U)$ and $\mathcal{O}(E',U)$, respectively. Then, we have the following relation:
$$\nabla_{\rho(l)}(\sum_j f_{ij}t_j)=\phi(\nabla_{\rho(l)} e_i)=0, \quad \forall l\in \Gamma(L_-, U).$$
This implies that $\rho(l)f_{ij}=0$ for all $l\in \Gamma(L_-, U)$. Using the fact that $L_-$ is maximal isotropic, we deduce that:
$$0=\rho(l)f_{ij}=\langle \rho(l),df_{ij}\rangle=( l, df_{ij})=( l, d_-f_{ij}).$$

Consequently, $f_{ij}$ must be a GH function. Therefore, $\phi$ is an isomorphism of GH vector bundles.

\end{proof}

\begin{Rem}
When $M$ constitutes a complex manifold, the  map mentioned in Lemma \ref{Lem:extmap} becomes a bijection. However, in the context of generalized complex manifolds, the conditions that would render this map bijective remain an open question for our future endeavors. This due to a generalized Newlander-Nirenberg theorem. We refer the reader to \cite{Wang-JGP2011} for a related result about this topic. 
\end{Rem}

Utilizing the property that a sequence of $C^\infty$ bundles admits a splitting, we can select a global $C^\infty$ splitting map $j\colon E_2\to E$ for the exact sequence \eqref{Seq:GHVB}, along with a corresponding projection $q\colon E\to E_1$. 
It is straightforward to verify that the composition $q\circ \dce\circ j$ exhibits $C^\infty(M)$-linearity, thereby qualifying as an element within the space $\Gamma\big(\rho(L_{-})^*\otimes\Hom(E_2,E_1)\big)$, where  $\dce$ encodes the $\rho(L_{-})$-module structure of $E$. Furthermore, $q\circ \dce\circ j$ emerges as a $1$-cocycle for the Lie algebroid $\rho(L_{-})$ with coefficients in the $\rho(L_{-})$-module $\Hom(E_2,E_1)$.

This observation enables us to define a map that assigns  each extension class determined by $E$ in ${\Ext}^1_{\mathcal{O}_M}(E_2,E_1)$ to the cohomology class $q\circ \dce\circ j$  in $H^1(\rho(L_{-}),\Hom(E_2,E_1))$. According to established literature, notably \cite{Huebschmann90} and \cite{CMP16}, this map is not merely well-defined but also bijective, establishing a correspondence between extension classes and cohomology classes.

\begin{lem}\cite{Huebschmann90,CMP16}
Let $E_1$ and $E_2$ be two GH vector bundles over a  GC manifold $M$.
There exists a  naturally isomorphism 
$$\nu\colon\Ext^1_{\mathcal{A}}(E_2,E_1)\to H^1(\rho(L_{-}), \Hom(E_2,E_1)).$$
\end{lem}

Below is our main result in this section.
\begin{Thm}\label{Thm:extensition-Liepair}
Let $E$ be a GH vector bundle over a regular GC manifold $M$. The natural injection 
$$\Psi\colon\Ext^1_{\mathcal{O}_M}(E,G^*M\otimes E)\to H^1(\rho(L_{-}), G^*M\otimes \End (E))$$
 maps the extension class $A(E)$ of the short exact sequence of GH vector bundles \eqref{seq:GHVB-jet} to the Lie algebroid cohomology class $[R^{\nabla}]$.
\end{Thm}

\begin{proof}
By Theorem 2.34 in \cite{CMP16}, the natural isomorphism 
$$\Ext^1_{\mathcal{A}}(E,G^*M\otimes E)\to H^1(\rho(L_{-}), G^*M\otimes \End (E))$$
maps the extension class of the short exact sequence \eqref{seq:GHVB-jet} of $\rho(L_{-})$-modules  to the Lie algebroid cohomology class $[R^{\nabla}]$. Therefore, to complete the argument, it suffices to verify that $\mathscr{J}^1 E$ and $\mathfrak{J}^1 E$  are isomorphic as $\rho(L_{-})$-modules. According to Proposition \ref{jetisomorphism}, this desired result is achieved.
\end{proof}

Applying Theorem \ref{Thm:extensition-Liepair}, we recover a result of Lang-Jia-Liu \cite{DGA2023}:
\begin{Cor}
Let $E$ be a GH vector bundle over a regular GC manifold $M$. Then the Atiyah class for the Lie pair $(T_\C M, \rho(L_-))$ vanishes if and only if the short exact sequence \eqref{seq:GHVB-jet} splits.
\end{Cor}

\vspace{20pt}  

\noindent \textbf{Ackonowledgments.} The author would like to thank Professor Zhangju Liu for fruitful discussions and useful comments.

\vspace{20pt}  

\begin{bibdiv}
	\begin{biblist}
		
		\parskip = 0.9em 

    \bib{Abouzaid-Boyarchenko}{article}{
   author={Abouzaid, M.},
   author={Boyarchenko, M.},
   title={Local structure of generalized complex manifolds},
   journal={J. Symplectic Geom.},
   volume={4},
   date={2006},
   number={1},
   pages={43--62},
   issn={1527-5256},
   review={\MR{2240211}},
}
        
        \bib{Atiyah}{article}{
	author={Atiyah, M.},
	title={Complex analytic connections in fibre bundles},
	journal={Trans. Amer. Math. Soc.},
	volume={85},
	date={1957},
	pages={181--207},
	issn={0002-9947},
	review={\MR{86359}},
}


\bib{Bailey}{article}{
   author={Bailey, M.},
   title={Local classification of generalized complex structures},
   journal={J. Differential Geom.},
   volume={95},
   date={2013},
   number={1},
   pages={1--37},
   issn={0022-040X},
   review={\MR{3128977}},
}
	

\bib{BCSX}{article}{
			author={Bandiera, R.},
			author={Chen, Z.},
			author={Sti$\acute{e}$non, M.},
			author={Xu, P.},
			title={Shifted Derived Poisson Manifolds Associated with Lie Pairs},
			journal={Comm. Math. Phys.},
			volume={375},
			date={2020},
			number={3},
			pages={1717--1760},
			review={\MR{4091493}},
		}

\bib{CLangX}{article}{
   author={Chen, Z.},
   author={Lang, H.},
   author={Xiang, M.},
   title={Atiyah classes of strongly homotopy Lie pairs},
   journal={Algebra Colloq.},
   volume={26},
   date={2019},
   number={2},
   pages={195--230},
   issn={1005-3867},
   review={\MR{3947420}},
}

\bib{CLX}{article}{
   author={Chen, Z.},
   author={Liu, Z.},
   author={Xiang, M.},
   title={Kapranov's construction of sh Leibniz algebras},
   journal={Homology Homotopy Appl.},
   volume={22},
   date={2020},
   number={1},
   pages={141--165},
   issn={1532-0073},
   review={\MR{4031996}},
}

    
		

        \bib{CMP16}{article}{
   author={Chen, Z.},
   author={Sti\'{e}non, M.},
   author={Xu, P.},
   title={From Atiyah classes to homotopy Leibniz algebras},
   journal={Comm. Math. Phys.},
   volume={341},
   date={2016},
   number={1},
   pages={309--349},
   issn={0010-3616},
   review={\MR{3439229}},
}

        
        \bib{Grandini-Poon-Rolle}{article}{
   author={Grandini, D.},
   author={Poon, Y.},
   author={Rolle, B.},
   title={Differential Gerstenhaber algebras of generalized complex
   structures},
   journal={Asian J. Math.},
   volume={18},
   date={2014},
   number={2},
   pages={191--218},
   issn={1093-6106},
   review={\MR{3217633}},
}

        \bib{Gualtieri}{article}{
   author={Gualtieri, M.},
   title={Generalized complex geometry},
   journal={Ann. of Math. (2)},
   volume={174},
   date={2011},
   number={1},
   pages={75--123},
   issn={0003-486X},
   review={\MR{2811595}},
}
		
		\bib{Hitchin03}{article}{
   author={Hitchin, N.},
   title={Generalized Calabi-Yau manifolds},
   journal={Q. J. Math.},
   volume={54},
   date={2003},
   number={3},
   pages={281--308},
   issn={0033-5606},
   review={\MR{2013140}},
}

\bib{Hitchin11}{article}{
   author={Hitchin, N.},
   title={Generalized holomorphic bundles and the $B$-field action},
   journal={J. Geom. Phys.},
   volume={61},
   date={2011},
   number={1},
   pages={352--362},
   issn={0393-0440},
   review={\MR{2747007}},
}
	

\bib{Hong}{article}{
   author={Hong, W.},
   title={Atiyah classes of Lie bialgebras},
   journal={J. Lie Theory},
   volume={29},
   date={2019},
   number={1},
   pages={263--275},
   issn={0949-5932},
   review={\MR{3904794}},
}

\bib{Huebschmann90}{article}{
   author={Huebschmann, J.},
   title={Poisson cohomology and quantization},
   journal={J. Reine Angew. Math.},
   volume={408},
   date={1990},
   pages={57--113},
   issn={0075-4102},
   review={\MR{1058984}},
}

\bib{Huybrechts}{book}{
   author={Huybrechts, D.},
   title={Complex geometry},
   series={Universitext},
   note={An introduction},
   publisher={Springer-Verlag, Berlin},
   date={2005},
   pages={xii+309},
   isbn={3-540-21290-6},
   review={\MR{2093043}},
}

\bib{Kapranov}{article}{
   author={Kapranov, M.},
   title={Rozansky-Witten invariants via Atiyah classes},
   journal={Compositio Math.},
   volume={115},
   date={1999},
   number={1},
   pages={71--113},
   issn={0010-437X},
   review={\MR{1671737}},
}



    
		\bib{Kontsevich}{article}{
   author={Kontsevich, M.},
   title={Deformation quantization of Poisson manifolds},
   journal={Lett. Math. Phys.},
   volume={66},
   date={2003},
   number={3},
   pages={157--216},
   issn={0377-9017},
   review={\MR{2062626}},
}
	
		
		

        	\bib{DGA2023}{article}{
			author={Lang, H.},
			author={Jia, X.},
			author={Liu, Z.},
			title={The Atiyah class of generalized holomorphic vector bundles},
			journal={Differential Geom. Appl.},
			volume={90},
			date={2023},
			pages={Paper No. 102031, 25},
			issn={0926-2245},
			review={\MR{4598904}},
		}
		
		\bib{LWX}{article}{
   author={Liu, Z.},
   author={Weinstein, A.},
   author={Xu, P.},
   title={Manin triples for Lie bialgebroids},
   journal={J. Differential Geom.},
   volume={45},
   date={1997},
   number={3},
   pages={547--574},
   issn={0022-040X},
   review={\MR{1472888}},
}

\bib{MSX}{article}{
   author={Mehta, R.},
   author={Sti\'{e}non, M.},
   author={Xu, P.},
   title={The Atiyah class of a dg-vector bundle},
   language={English, with English and French summaries},
   journal={C. R. Math. Acad. Sci. Paris},
   volume={353},
   date={2015},
   number={4},
   pages={357--362},
   issn={1631-073X},
   review={\MR{3319134}},
}

\bib{Molino}{article}{
   author={Molino, P.},
   title={Classe d'Atiyah d'un feuilletage et connexions transverses
   projetables},
   language={French},
   journal={C. R. Acad. Sci. Paris S\'{e}r. A-B},
   volume={272},
   date={1971},
   pages={A779--A781},
   issn={0151-0509},
   review={\MR{281224}},
}

\bib{Poon2021}{article}{
   author={Poon, Y.},
   title={Abelian complex structures and generalizations},
   journal={Complex Manifolds},
   volume={8},
   date={2021},
   number={1},
   pages={247--266},
   review={\MR{4303250}},
}

		
				
\bib{Quantum-groupoids}{article}{
   author={Xu, P.},
   title={Quantum groupoids},
   journal={Comm. Math. Phys.},
   volume={216},
   date={2001},
   number={3},
   pages={539--581},
   issn={0010-3616},
   review={\MR{1815717}},
}




\bib{Wang-JGP2011}{article}{
   author={Wang, Y.},
   title={Generalized holomorphic structures},
   journal={J. Geom. Phys.},
   volume={61},
   date={2011},
   number={10},
   pages={1976--1984},
   issn={0393-0440},
   review={\MR{2822464}},
}

\bib{MR3157903}{article}{
   author={Wang, Y.},
   title={Deformations of generalized holomorphic structures},
   journal={J. Geom. Phys.},
   volume={77},
   date={2014},
   pages={72--85},
   issn={0393-0440},
   review={\MR{3157903}},
}


	
		
	\end{biblist}
\end{bibdiv}
\end{document}